\documentclass[letterpaper, 12pt]{article} 
\usepackage{geometry} 
\pdfoutput=1
\usepackage{amsmath,amsfonts,amssymb,amsthm,hyperref,xy,tikz,mathtools,float}
\usepackage{verbatim}
\usetikzlibrary{arrows,petri,positioning,automata}
\usetikzlibrary{decorations.pathreplacing}

\xyoption{all}
\xyoption{pdf}

\newcommand{\id}{\mathsf{id}} 
\newcommand{\cod}{\mathrm{codim\,}} 
\newcommand{\Lie}{\mathbf{Lie}} 
\newcommand{\Mon}{\mathrm{Mon}} 
\newcommand{\del}[1]{{\tfrac{\partial}{\partial #1 }}} 
\newcommand{\Tr}{\mathrm{Tr}} 
\newcommand{\PP}{\mathbb{P}} 
\newcommand{\RR}{\mathbb{R}} 
\newcommand{\ZZ}{\mathbb{Z}} 
\newcommand{\NN}{\mathbb{N}} 
\newcommand{\Aa}{\mathcal{A}} 
\newcommand{\Ff}{\mathcal{F}} 
\newcommand{\Gg}{\mathcal{G}} 
\newcommand{\Hh}{\mathcal{H}} 

\newcommand{\Oo}{\mathcal{O}}
\newcommand{\Nn}{\mathcal{N}} 
\newcommand{\Kk}{\mathcal{K}} 
\newcommand{\Gpd}{\mathbf{Gpd}} 
\newcommand{\Lint}{\mathbf{Gpd}} 
\newcommand{\Lnorm}{\boldsymbol\Lambda}
\newcommand{\Del}{\partial}
\newcommand{\til}[1]{\widetilde{#1}}

\newcommand{\defw}[1]{\emph{#1}}
\newcommand{\blowup}[2]{\mathsf{Bl}_{#2}(#1)}
\newcommand{\BGpd}[2]{[#1\!:\!#2]}
\newcommand{\BAlg}[2]{[#1\!:\!#2]} 
\newcommand{\tot}{\mathrm{tot}} 

\newcommand{\Sym}{\mathrm{Sym}}
\newcommand{\sD}{\texttt{D}}
\newcommand{\sV}{\texttt{V}}
\newcommand{\sE}{\texttt{E}}

\newcommand{\sH}{\texttt{H}}
\newcommand{\Pair}{\mathrm{Pair}}
\newcommand{\At}{\mathbf{At}}
\newcommand{\Hol}{\mathcal{H}\!{\mathit{ol}}}
\newcommand{\R}{\mathbf{R}}
\newcommand{\fF}{\mathfrak{f}}
\newcommand{\gG}{\mathfrak{g}}
\newcommand{\hH}{\mathfrak{h}}
\newcommand{\kK}{\mathfrak{k}}

\newcommand{\rra}{\rightrightarrows}
\newcommand{\Bl}[2]{\mathrm{Bl}_{#2}(#1)}

\newcommand{\genhole}[1]{\draw[#1, xscale=.2, yscale=.3] (-2.2,.4) .. controls (-1.5,-0.3) and (-1,-0.5) .. (0,-.5) .. controls (1,-0.5) and (1.5,-0.3) .. (2.2,0.4);
\draw[#1,xscale=.2, yscale=.3] (-1.75,0) .. controls (-1.5,0.3) and (-1,0.5) .. (0,.5) .. controls (1,0.5) and (1.5,0.3) .. (1.75,0);
}
\usetikzlibrary{decorations.markings}
\tikzset{middlearrow/.style={
        decoration={markings,
            mark= at position 0.5 with {\arrow{#1}} ,
        },
        postaction={decorate}
    }
}

\makeatletter
\newcommand{\thickbar}{\mathpalette\@thickbar}
\newcommand{\@thickbar}[2]{{#1\mkern1.5mu\vbox{
  \sbox\z@{$#1\mkern-1mu#2\mkern-1mu$}%
  \sbox\tw@{$#1\overline{#2}$}%
  \dimen@=\dimexpr\ht\tw@-\ht\z@-.6\p@\relax
  \hrule\@height.4\p@ 
  \vskip1\p@
  \hrule\@height.4\p@ 
  \vskip\dimen@
  \box\z@}\mkern1.5mu}
}
\makeatother

\newtheorem{theorem}{Theorem}[section]
\newtheorem{corollary}[theorem]{Corollary}
\newtheorem{lemma}[theorem]{Lemma}
\newtheorem{prop}[theorem]{Proposition}
\theoremstyle{definition}
\newtheorem{definition}[theorem]{Definition}
\newtheorem{remark}[theorem]{Remark}
\newtheorem{example}[theorem]{Example}
\numberwithin{equation}{section}

\title{\vspace{-2em}\bf Symplectic groupoids of log symplectic manifolds}
\author{Marco Gualtieri\footnote{University of Toronto;  mgualt@math.toronto.edu.} \ \  and Songhao Li\footnote{University of Toronto;  sli@math.toronto.edu.}}
\date{}
\usepackage{color}
\definecolor{tocolor}{rgb}{.1,.1,.5}
\definecolor{urlcolor}{rgb}{.2,.2,.6}
\definecolor{linkcolor}{rgb}{.1,.1,.6}
\definecolor{citecolor}{rgb}{.6,.2,.1}
\hypersetup{colorlinks=true, urlcolor=urlcolor, linkcolor=linkcolor, citecolor=citecolor}

\begin{document}
\maketitle 

\abstract{
A log symplectic manifold is a Poisson manifold which is generically nondegenerate.
We develop two methods for constructing the symplectic groupoids of log symplectic manifolds.
The first is a blow-up construction, corresponding to the notion of an elementary modification of a Lie algebroid along a subalgebroid.  The second is a gluing construction, whereby groupoids defined on the open sets of an appropriate cover may be combined to obtain global integrations.  This allows us to classify all Hausdorff symplectic groupoids of log symplectic manifolds in a combinatorial fashion, in terms of a certain graph of fundamental groups associated to the manifold.  Using the same ideas, and as a first step, we also construct and classify the groupoids integrating the Lie algebroid of vector fields tangent to a smooth hypersurface.
}

\begingroup
\hypersetup{linkcolor=tocolor}
\setcounter{tocdepth}{2}
\tableofcontents\pagebreak
\endgroup

\section{Introduction}\label{intr}

	A Poisson manifold $M$ may have quite complicated local behaviour, as it involves a singular foliation by smooth symplectic leaves of varying dimension.  For this reason, it is natural to consider symplectic manifolds $S$ which map surjectively to $M$ via a Poisson map.  In this way, one hopes to replace the study of $M$ by the study of these symplectic ``realizations'' $S$.  This programme was initiated by Weinstein~\cite{MR866024}, who observed that in many cases a canonical realization $\Gg$ may exist, having twice the dimension of $M$ and further endowed with the structure of a Lie groupoid over $M$.  In analogy with the integration of a Lie algebra to a Lie group, $\Gg$ is called the symplectic groupoid integrating the Poisson manifold $M$. 

In his original paper, Weinstein observed that a Poisson manifold $M$, and more generally, a Lie algebroid, may fail to integrate to a smooth Lie groupoid.  Since that time, increasingly powerful general theories have been developed to address the question of existence of integrations of Lie algebroids, culminating in the work of Crainic and Fernandes~\cite{MR1973056,MR2128714}, who described the obstruction theory for the existence of integrations, using a general construction of the symplectic groupoid given by Cattaneo and Felder~\cite{MR1938552} in terms of an infinite--dimensional symplectic quotient. 

In view of the fact that explicit examples of symplectic groupoids are not very numerous, 
and since a concrete understanding of their topology is desirable in many applications, such as in the theory of geometric quantization, the purpose of this paper is to develop methods for the explicit construction and classification of symplectic groupoids.  We apply these methods to 
study the symplectic groupoids of log symplectic manifolds, which are generically nondegenerate Poisson manifolds that drop rank along a smooth hypersurface.  Log symplectic surfaces were completely classified by Radko; she called them topologically stable Poisson structures~\cite{MR1959058}.  In general dimension, the behaviour of these Poisson structures is carefully described by Guillemin, Miranda, and Pires~\cite{MR2861781,Guillemin2}, who call them {\it b}--symplectic manifolds.
 
The first method is inspired by work of Weinstein~\cite{MR1394388}, Mazzeo--Melrose~\cite{MR1734130}, and Monthubert~\cite{MR1600121}, and involves a systematic use of the projective blow-up operation.  By 
understanding how the blow-up affects the Lie algebroid, Lie groupoid, and Poisson structures individually, we are able to construct explicitly the \emph{adjoint} symplectic groupoids of log symplectic manifolds.   The existence of these groupoids has been known since the work of Debord~\cite{MR1906783}, but their detailed global geometry has not been fully explored.  

The second method is a gluing construction, inspired by the work of Nistor~\cite{MR1774632}.  
By choosing an appropriate open cover, called an \emph{orbit cover}, on $M$, we show that it is possible to give an 
explicit combinatorial description of the category of all integrations of the log tangent bundle, 
as well as the category of all Hausdorff symplectic groupoids integrating a proper
log symplectic manifold.   

We thank Henrique Bursztyn, Marius Crainic, Rui Fernandes, Lisa Jeffrey, Eva Miranda, Ana Rita Pires, Alexander Polishchuk, Brent Pym and Alan Weinstein for helpful discussions and insights.  This research is supported by an NSERC Discovery Grant, an Ontario ERA and an Ontario Graduate Scholarship.

	\subsection{Log symplectic manifolds}\label{logsymp}

	Poisson manifolds of log symplectic type were studied by Goto in the holomorphic category~\cite{MR1953353}, and by Guillemin, Miranda and Pires in the smooth category~\cite{MR2861781,Guillemin2}.  In dimension 2, Radko provided a complete classification~\cite{MR1959058}.  These Poisson manifolds are generically symplectic, and degenerate along a  hypersurface.  In this section, we describe the Poisson geometry near such a hypersurface.
%

\begin{definition}\label{bsymp}
A \defw{log symplectic manifold} is a smooth $2n$--manifold $M$, equipped with a Poisson structure $\pi$ whose Pfaffian, $\pi^n$, vanishes transversely.  
\end{definition}
The degeneracy locus $D = (\pi^n)^{-1}(0)$ is then an embedded, possibly disconnected, Poisson hypersurface, and $M\backslash D$ is a union of open symplectic leaves.  If $M$ is compact, then both $D$ and $M\backslash D$ have finitely many components.  The Poisson structure $\pi$ is called log symplectic because $\pi^{-1}$ defines a logarithmic symplectic form, as we now explain.
\subsubsection{Log tangent bundle}
\begin{definition}\label{logtgt}
The \defw{log tangent bundle} $T_D M$ associated to a closed hypersurface $D\subset M$ is the vector bundle associated to the sheaf of vector fields on $M$ tangent to $D$.  Equipped with the induced Lie bracket and the inclusion morphism $a:T_DM\to TM$, it is a Lie algebroid.
\end{definition}
\begin{remark}
The de Rham complex of the Lie algebroid $T_D M$ may be interpreted as differential forms with logarithmic singularities along $D$; it was introduced in~\cite{MR0417174} and is denoted by $(\Omega^\bullet_M(\log D), d)$.\qed
\end{remark}
\begin{prop}
Let $(M,\pi)$ be a log symplectic manifold with degeneracy locus $D$. Then $\pi^{-1}$ determines a nondegenerate closed logarithmic 2-form in $\Omega^2_M(\log D)$. 
\end{prop}
\begin{proof}
If $f$ is a local smooth function vanishing to first order along $D$, then $\pi(df)$ must vanish along $D$, since $D$ is Poisson.  Therefore, $\pi(df)=fY$ for a smooth vector field $Y$, which must be tangent to $D$, since $Y(f) = f^{-1}\pi(df,df) = 0$.  This proves that $\pi:T^*M\to TM$ lifts to $\widetilde{\pi}:T^*_DM\to T_DM$, commuting with the natural inclusions:
\[
\xymatrix{
	T^*_DM\ar[r]^{\widetilde{\pi}} & T_DM\ar[d]^{a}\\
	T^*M\ar[r]_{\pi}\ar[u]^{a^*} & TM}
\]
It remains to show $\widetilde{\pi}$ is an isomorphism, but this is obtained from the determinant of the above diagram: $\det a$ and $\det a^*$ vanish to first order along $D$, whereas $\det \pi = \pi^n\otimes\pi^n$ vanishes to second order. Hence $\widetilde{\pi}$ is nondegenerate, and $\widetilde{\pi}^{-1}$ is closed since $\pi^{-1}$ is a well-defined symplectic form on $M\backslash D$.  
\end{proof}
\subsubsection{Poisson line bundles}
To describe the geometry of log symplectic manifolds in a neighbourhood of the degeneracy locus, we make use of the notion of a rank 1 Poisson module~\cite{MR1465521} or Poisson line bundle, which we now recall.
\begin{definition}\label{pvb}
A \defw{Poisson vector bundle} over the Poisson manifold $(M,\pi)$ is a vector bundle $V\to M$ equipped with a flat Poisson connection, i.e. a differential operator $\Del:\Gamma(V)\to\Gamma(TM\otimes V)$ such that $\Del(fs) = \pi(df)\otimes s + f\Del s$ for $f\in C^\infty(M)$ and with vanishing curvature in $\Gamma(\wedge^2 TM)$.
\end{definition}
A real line bundle $L$ always admits a flat connection $\nabla$, and any flat Poisson connection $\Del$ may be written 
\begin{equation}\label{flatdel}
\Del = \pi\circ \nabla + Z,
\end{equation}
for $Z$ a Poisson vector field.  Another flat connection $\nabla'$ differs from $\nabla$ by a closed real 1-form $A$, so that the Poisson vector field $Z$ is determined uniquely by $\Del$ only up to the addition of a locally Hamiltonian vector field.  For this reason, if the underlying Poisson manifold has odd dimension $2n-1$, the multivector $Z\wedge \pi^{n-1}$ is independent of the choice of $\nabla$.  We call this the residue of $(L,\Del)$, following~\cite{MGBP1}.
\begin{definition}\label{residef}
The \defw{residue} $\chi\in\Gamma(\wedge^{2n-1} TD)$ of a Poisson line bundle $(L,\Del)$ over a Poisson $(2n-1)$--manifold $(D,\sigma)$ is defined by 
\[
\chi = Z\wedge \sigma^{n-1},
\]
where $Z$ is a Poisson vector field given by~\eqref{flatdel}.
\end{definition}
\begin{example}\label{acanmod}
The anticanonical bundle $K^* = \det TM$ is a Poisson line bundle, with Poisson connection $\Del$ uniquely determined by the condition
\begin{equation}\label{canondef}
\Del_{df}(\rho) = \mathcal{L}_{\pi(df)}\rho,
\end{equation}
where $f\in C^\infty(M)$ and $\rho\in \Gamma(K^*)$.  In coordinates where $\pi=\pi^{ij}\partial_{x_i}\wedge\partial_{x_j}$, the associated Poisson vector field via~\ref{flatdel} is $Z = ({\partial_{k} \pi^{ik}})\partial_{x_i}$, known as the modular vector field~\cite{MR1484598} associated to the Lebesgue measure.
\end{example}
\subsubsection{Degeneracy loci}
We now describe the Poisson geometry of the degeneracy locus $D$, recovering some results of~\cite{MR2861781} by slightly different means.
\begin{prop}\label{rankloc}
The degeneracy locus of a log symplectic $2n$-manifold $(M,\pi)$ is a $(2n-1)$-manifold whose Poisson structure has constant rank $2n-2$ and which admits a Poisson vector field transverse to the symplectic foliation.  In particular, $D$ is unimodular.   
\end{prop}
\begin{proof}
Since $\pi^n$ vanishes transversely along $D$, its first derivative (i.e. first jet) $j^1(\pi^n)$ defines, along $D$, a nonvanishing section $\chi$ of $N^*D\otimes\wedge^{2n}TM|_D\cong \wedge^{2n-1} TD$, i.e. a covolume form on $D$.  On the other hand, the Leibniz rule gives $j^1(\pi^n) = n \pi^{n-1}\wedge j^1(\pi)$, so that $\pi^{n-1}$ is nonvanishing along $D$, showing $\pi$ has constant rank $2n-2$ along $D$.   

By Example~\ref{acanmod}, the anticanonical bundle $K^*=\det TM$ has a natural flat Poisson connection $\Del$.  Upon choosing a usual flat connection $\nabla$ on $K^*$, we obtain, via~\eqref{flatdel}, a Poisson vector field $Z$ on $M$, which must be tangent to the degeneracy locus $D$.  It remains to show that $Z$ is transverse to the symplectic leaves on $D$.  This may be rephrased as follows: the anticanonical bundle restricts to a Poisson line bundle on $D$, canonically the normal bundle $ND$, and we claim its residue $Z\wedge\pi^{n-1}$ is nonvanishing.  In fact, we show it coincides with the covolume $\chi$ defined above.

To see this, choose a flat local trivialization $\rho$ for $K^*$ near a point in $D$.  Then $\pi^n = f\rho$, for a smooth function $f$.  From~\eqref{canondef}, we have $\Del\pi^n = 0$, and applying~\eqref{flatdel}, we obtain $\pi(df) + f Z = 0$, i.e. $Z$ has singular Hamiltonian $-\log f$. Therefore:
\[
\chi = \Tr(\nabla \pi^n)|_D = \Tr(\nabla (f\rho))|_D = (i_{d\log f} \pi^n)|_D =  n(Z\wedge \pi^{n-1})|_D,
\]
showing that $Z\wedge\pi^{n-1}$ is nonvanishing on $D$, as required.  Also, since $Z$ is Poisson, this covolume form is invariant under Hamiltonian flows, i.e. $D$ is unimodular.
\end{proof}
Proposition~\ref{rankloc} has a converse, because the total space of a Poisson line bundle is naturally Poisson~\cite{MR1465521}.  This provides an alternative approach to the extension theorem for regular corank one Poisson structures in~\cite{Guillemin2}.
\begin{prop}\label{totpois}
Let $(D,\sigma)$ be a Poisson $(2n-1)$--manifold of constant rank $2n-2$, and let $N$ be a Poisson line bundle with nonvanishing residue $\chi\in \Gamma(\wedge^{2n-1} TD)$.  Then the total space of $N$ is a log symplectic manifold with degeneracy locus $(D,\sigma)$.
\end{prop}
\begin{proof}
Choose a flat connection\footnote{Proposition~\ref{totpois} also holds for complex line bundles~\cite{MR1465521}, when flat connections may not exist; for later convenience we present an argument tailored to the real case.} $\nabla$ on $N$, and let $Z$ be given by~\eqref{flatdel}.  Let $\widetilde{\sigma}$ and $\widetilde{Z}$ be the horizontal lifts of $\sigma$ and $Z$ to $\tot(N)$, and let $E$ be the Euler vector field.  Then 
\begin{equation}\label{totpi}
\pi = \widetilde{\sigma} + \widetilde{Z} \wedge E
\end{equation}
is a log symplectic structure on $\tot(N)$ with degeneracy locus $(D,\sigma)$.  We now verify that $\pi$ is independent of the choice of $\nabla$: for another connection $\nabla'=\nabla + A$, the horizontal lifts differ by $\widetilde{X}' - \widetilde{X} = A(X)E$, so that 
\[
\widetilde{\sigma}' = \widetilde{\sigma} + \sigma(A)\wedge E.
\]
On the other hand, for the new connection $Z' = Z - \sigma(A)$, so that 
\[
\widetilde{Z}' \wedge E =\widetilde{Z} \wedge E -\sigma(A)\wedge E,
\]
showing~\eqref{totpi} is independent of $\nabla$.
\end{proof}
\subsubsection{Proper log symplectic manifolds}\label{sympmaptorus}
In the proof of Proposition~\ref{rankloc}, we saw that the normal bundle of the degeneracy locus is itself a Poisson line bundle.  By Proposition~\ref{totpois}, therefore, its total space inherits a natural log symplectic structure.  This structure is called the \defw{linearization} of $\pi$ along $D$.  We require a very concrete description of this linearized Poisson structure, so we restrict to a special class of log symplectic manifolds.
\begin{definition}
A log symplectic manifold is \defw{proper} when each connected component $D_j$ of its degeneracy locus $D = \coprod_j D_j$ is compact and contains a compact symplectic leaf $F_j$. 
\end{definition}
Now let $(M,\pi)$ be a proper log symplectic manifold, with connected degeneracy locus $D$, containing a compact symplectic leaf $(F,\omega)$. As shown in~\cite{MR2861781}, it follows from the Reeb--Thurston stability theorem that the transverse Poisson vector field $Z$ renders $D$ isomorphic to a symplectic mapping torus $S^1_\lambda\ltimes_\varphi F$: 
\begin{equation}\label{maptor}
S^1_\lambda\ltimes_\varphi F = \frac{F \times \RR}{(x,t)\sim (\varphi(x), t+ \lambda)}, \ \ \ \lambda>0,
\end{equation}
where $\varphi:F\to F$ is a symplectomorphism.  Note that there is a natural projection map 
\[
f:S^1_\lambda\ltimes_\varphi F \longrightarrow S^1_\lambda = \frac{\RR}{t\sim t+\lambda},
\] 
and the Poisson structure on the mapping torus is given by $\iota\omega^{-1}\iota^*$, where $\iota$ is the inclusion morphism of the subbundle $\ker (Tf)\subset TD$.  The normal bundle of $D$ is a real line bundle, classified up to isomorphism by $H^1(D,\ZZ_2)=H^1(F,\ZZ_2)^\varphi\times \ZZ_2$, where $H^1(F,\ZZ_2)^\varphi$ denotes the subgroup of $\varphi$--invariant classes.  In other words, the normal bundle is isomorphic to a tensor product $N=\widetilde{L} \otimes f^*Q$, where $\widetilde{L}$ is the line bundle induced on the mapping torus from a $\ZZ$--equivariant line bundle $L$ on $F$, and $Q$ is a line bundle on $S^1_\lambda$.  Choosing a $\ZZ$--invariant flat connection on $L$ and a flat connection on $Q$, we obtain a flat connection $\nabla$ on $N$.  The Poisson module structure on $N$ is then simply 
\[
\Del = \pi\circ\nabla + \partial_t,
\]
and so the residue $\chi$ of the Poisson line bundle is $(\omega^{n-1}\wedge f^*dt)^{-1}$. The constant $\lambda$ retained in the construction has an invariant meaning: it is the ratio of the volume of $D$ (with respect to $\chi^{-1}$) to the volume of the symplectic leaf $F$ (with respect to $\omega^{n-1}$).  Summarizing the above discussion, we obtain the following result.
\begin{prop} \label{Prop: LogSympLineariation}
The linearization of a proper log symplectic $2n$--manifold along a connected component $D$ of its degeneracy locus is classified up to isomorphism by the following data: a compact symplectic $(2n-2)$--manifold $(F,\omega)$, a symplectomorphism $\varphi$, a cohomology class in $H^1(F,\ZZ_2)^\varphi\times \ZZ_2$, and a positive real number $\lambda$, called the modular period.  
\end{prop}

One of the main results of~\cite{Guillemin2}, extending the result in~\cite{MR1959058} for surfaces, is a proof, via a Moser-type deformation argument, that log symplectic manifolds are linearizable, namely that a tubular neighbourhood of each component $D_j$ of the degeneracy locus is isomorphic, as a log symplectic manifold, to a neighbourhood of the zero section in the linearization along $D_j$.
\begin{theorem}[Guillemin--Miranda--Pires~\cite{Guillemin2}]\label{Thm: GMP} A log symplectic manifold is linearizable along its degeneracy locus.\end{theorem}
%
%
%
\begin{example} \label{Example: EllipticCurve}
The cubic polynomial $g(x) = x(x-1)(x-t)$,  $0<t<1$, defines a Poisson structure on $\RR^2$ given by 
\[
\pi = (g(x)-y^2) \partial_x\wedge \partial_y,
\] 
which extends smoothly to a log symplectic structure on $\RR P^2$ with degeneracy locus $D$ given by the real elliptic curve $y^2 = g(x)$, as shown below.
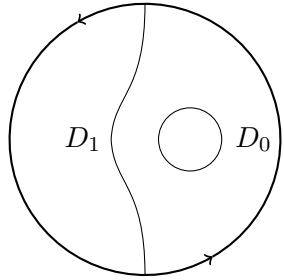
\begin{center}
	\begin{tikzpicture}[scale=.6][>=angle 60]
	\draw	[->, thick] (1.5,-2.598) arc (-60:120:3cm);
	\draw [->, thick] (-1.5,2.598) arc (120:300:3cm);
	\draw (0,-3) .. controls (0,-1) and (-.7,-1) .. (-.75,0) .. controls (-.7,1) and (0,1) .. (0,3) node [at start, left] {$D_1$};
	\draw (1,0) circle (0.7cm) node [right=1.2em] {$D_0$};
	\end{tikzpicture}
\end{center}
The degeneracy locus has two connected components: $D_0$, containing $\{(0,0),(t,0)\}$ and with trivial normal bundle, and $D_1$, containing $\{(1,0), (\infty,0)\}$ and with nontrivial normal bundle. The residue $\chi$ of $\pi$ along $D$ is such that $\chi^{-1}$ coincides with the Poincar\'e residue
\[
\frac{dx}{2y} = \frac{dx}{2\sqrt{x(x-1)(x-t)}}.
\] 
This extends to a holomorphic form on the complexified elliptic curve, in which $D_0, D_1$ are cohomologous, so that the modular periods $\lambda_0, \lambda_1$ of $D_0, D_1$ must coincide.  The modular period $\lambda_0$ is therefore a classical elliptic period~\cite{MR2024529}, given by the Gauss hypergeometric function 
\[
\lambda_0(t) = \pi F(\tfrac{1}{2}, \tfrac{1}{2}, 1; t).
\]
\end{example}

	\subsection{Lie groupoids and Lie algebroids}

	 In this paper, the manifold $M$ of objects of a Lie groupoid is assumed to be a smooth manifold in the usual sense, but the space of arrows $\Gg$ must satisfy all smooth manifold axioms {\emph{except the Hausdorff requirement}}.  This relaxation is customary in the literature because of the fact that smooth Lie algebroids on $M$ often integrate to smooth, but non-Hausdorff, Lie groupoids (see, for example,~\cite{MR1973056, MR2012261}). One of the key points of this paper is that we identify precisely which of the groupoids we construct are actually Hausdorff.  For more details concerning the separation axiom for Lie groupoids, see~\cite{MR2592728}.
 
 \subsubsection{Notation}
A Lie groupoid $(\Gg, M, s, t, m, \id)$, or $\Gg\rightrightarrows M$ for short, is defined as follows.  The space of arrows $\Gg$ is a smooth but possibly non-Hausdorff manifold, equipped with smooth submersions $s, t$ to the manifold of objects $M$, called the \defw{source} and \defw{target} maps, respectively.  The smooth multiplication map $m$ is defined on the manifold of composable arrows, defined by the fiber product 
\[
\Gg^{(2)} \coloneqq \Gg {_{t}\times_s} \Gg = \{(g, h) \in \Gg \times \Gg ~|~ t(g) = s(h) \},
\]
and the identity arrow for each object is given by a smooth embedding $\id$, so that the groupoid structure maps may be displayed as follows:
\[
\xymatrix{
\Gg^{(2)} \ar[r]^-{m}		& \Gg \ar@/_1.2pc/[r]_-{t} \ar@/^1.2pc/[r]^-{s}	& M \ar[l]_{\id}}.
\]
These maps satisfy the expected compatibility conditions for a category where all morphisms have inverses.

For $x,y\in M$, the space $s^{-1}(x)\cap t^{-1}(y)$ of arrows from $x$ to $y$ is denoted by $\Gg(x,y)$ and, in analogy with group actions, $\Gg_x \coloneqq \Gg(x,x)$ is called the \defw{isotropy group} at $x$. Also, as in the case of a group action, the manifold $M$ obtains an equivalence relation
\[
x\sim y  \Leftrightarrow y\in t(s^{-1}(x)),
\]  
partitioning $M$ into equivalence classes called \defw{orbits} of the groupoid.

\subsubsection{Lie groupoid quotients}

We say that a subgroupoid $\Nn\rightrightarrows L$ of the Lie groupoid $\Gg\rightrightarrows M$ is a \emph{Lie subgroupoid} when the inclusions $\Nn\subset \Gg$ and $L\subset M$ are smooth embeddings.  A subgroupoid $\Nn\rightrightarrows L$ of $\Gg\rightrightarrows M$ is called \emph{wide} when $L = M$; such a wide subgroupoid is \emph{normal} if for all $g\in \Gg(x,y)$, 
\[
g \Nn_x g^{-1} = \Nn_y.
\]
A normal subgroupoid defines an equivalence relation $R\subset \Gg\times\Gg$ via 
\begin{equation}\label{eqrel}
R \coloneqq \{(g, g') \in \Gg \times \Gg ~|~ \Nn g = \Nn g'\},
\end{equation}
whose equivalence classes are the right cosets $\Gg / \Nn$.  Finally, we say that the subgroupoid is \emph{totally disconnected} when $\Nn(x,y)$ is empty for $x\neq y$.  We now describe the quotient construction for Lie groupoids, following \cite[Theorem 3.3]{MR1078178}, which treats Hausdorff Lie groupoids but is easily extended to the general case:
\begin{theorem}[Higgins-Mackenzie~\cite{MR1078178}] \label{Prop: LieGpdQuotient}
Let $\Gg \rightrightarrows M$ be a Lie groupoid and $\Nn\rightrightarrows M$ a totally disconnected normal Lie subgroupoid.  Then the quotient $\Gg / \Nn\rightrightarrows M$ is a Lie groupoid. Furthemore, if $\Gg$ is Hausdorff, then the quotient groupoid is Hausdorff if and only if $\Nn$ is closed in $\Gg$.
\end{theorem}
\begin{proof}[Sketch of proof:]  The fact that $\Nn$ is normal ensures that $\Gg/\Nn$ is a groupoid, and since $\Nn$ is totally disconnected, the quotient groupoid has the same space of objects $M$ as $\Gg$, as described in~\cite{MR1078178}.  

To obtain smoothness of $\Gg/\Nn$, we use a theorem of Godement~\cite[Theorem II.3.12.2]{MR2179691}, which states that if $X$ is a possibly non-Hausdorff smooth manifold, and $R\subset X\times X$ is an equivalence relation, then $X/R$ is a possibly non-Hausdorff smooth manifold if and only if $R$ is a wide Lie subgroupoid of the pair groupoid $X\times X$.  Furthermore, if $X$ is Hausdorff, then $X/R$ is Hausdorff if and only if $R\subset X\times X$ is closed. 

To apply this to the case at hand, note that $\Nn$ is smooth, embedded and closed if and only if the graph of its equivalence relation $R$ in~\eqref{eqrel} has the respective property. Applying Godement's result to $X=\Gg$, we obtain the required smoothness of the quotient, as well as the Hausdorff condition.     The verification that the groupoid operations are smooth is identical to that in~\cite{MR1078178}.
\end{proof}

\subsubsection{The Lie functor}\label{liefunctor}
Just as a Lie group determines a Lie algebra, the Lie groupoid $\Gg\rightrightarrows M$ determines a Lie algebroid $A$, which is a vector bundle over $M$ equipped with a Lie bracket on its sections as well as a bracket-preserving bundle map $a:A\to TM$, satisfying the Leibniz rule
\[
[X,fY] = f[X,Y] +  a(X)(f)Y,
\]
for all sections $X, Y$ of $A$ and functions $f$ on $M$.  The functor which associates a Lie algebroid to a Lie groupoid, is defined as follows (see~\cite{MR0216409} for details):

\begin{definition} 
The \emph{Lie functor} associates, to any Lie groupoid $\Gg\rightrightarrows M$, the Lie algebroid
\[
\Lie(\Gg) = \id^*\ker\left(Ts:T\Gg \to TM\right),
\]
with morphism to $TM$ given by the restriction of the derivative $Tt$ to $\ker (Ts)$, and bracket defined by the Lie bracket on left-invariant vector fields.  For any groupoid homomorphism $\Psi:\Gg\to \Gg'$, we have the induced morphism of Lie algebroids 
\[
\Lie(\Psi) = T\Psi|_{\Lie(\Gg)}: \Lie(\Gg)\to \Lie(\Gg').
\] 
If a Lie algebroid $A$ is equipped with an isomorphism $A\cong \Lie(\Gg)$, we say that $\Gg$ is an \emph{integration} of, or \emph{integrates}, $A$. 
\end{definition}

The main purpose of this paper is to construct and classify Lie groupoids integrating two types of Lie algebroids.  The first is the log tangent bundle of a hypersurface, as in Definition~\ref{logtgt}.  The second is the Lie algebroid of a log symplectic structure, which is a special case of the Lie algebroid defined by any Poisson structure:
\begin{definition}\label{poisalg}
The Lie algebroid $T^*_\pi M$ of a Poisson manifold $(M,\pi)$ is the cotangent bundle $T^*M$, equipped with the \emph{Koszul bracket}
\[
[\alpha,\beta] = L_{\pi(\alpha)}\beta - L_{\pi(\beta)}\alpha - d\pi(\alpha,\beta),
\]
as well as the bundle map $\pi:T^*M\to TM$.
\end{definition}
The tautological symplectic form on $T^*M$ then endows any integration of $T^*_\pi M$ with a symplectic form compatible with the groupoid structure -- for this reason, the integrations of $T^*_\pi M$ are called \emph{symplectic groupoids}.  It will be useful to view symplectic groupoids as a special case of Poisson groupoids, which we now define. 

\begin{definition} \label{def: PoissonAgd}
A \defw{Poisson groupoid} is a Lie groupoid $\Gg$, equipped with a Poisson structure $\Pi$, such that the graph of multiplication map
\[
\Gamma_m = \{(g, h, m(g,h)) \in \Gg \times \Gg \times {\Gg} \}
\]
is coisotropic with respect to $\Pi \oplus \Pi \oplus (-\Pi)$.  In the case that $\Pi$ is non-degenerate, $\Omega = \Pi^{-1}$ defines a symplectic form such that $\Gamma_m$ is Lagrangian, and then $(\Gg, \Pi)$ is called a \defw{symplectic groupoid}. 
\end{definition}

If $(\Gg, \Omega)$ is a symplectic groupoid over $M$, then the Poisson structure $\Omega^{-1}$ is invariant along the source and target fibers, so that it descends to a Poisson structure $\pi \coloneqq Ts(\Omega^{-1}) = - Tt(\Omega^{-1})$ on $M$.  The groupoid $\Gg\rightrightarrows M$ is then an integration of the Poisson algebroid $T^*_\pi M$. For this reason, symplectic groupoids provide a special class of symplectic realizations of Poisson manifolds, as discussed in~\S\ref{intr}.

	\subsection{The category of integrations} \label{sub: IntCat}

	A finite-dimensional Lie algebra $\mathfrak{g}$ determines a lattice $\Lambda(\mathfrak{g})$ of connected Lie groups which integrate it: $G'$ covers $G$ in $\Lambda(\mathfrak{g})$ if there is a morphism $G'\to G$ inducing the identity map on $\mathfrak{g}$.  The initial object of this lattice is the simply-connected integration $\widetilde G$; all other groups in the lattice are quotients of $\widetilde G$ by discrete subgroups of its center; and the terminal object, when it exists, is called the \emph{adjoint form} of the group.

When a Lie algebroid is integrable, its lattice, or more properly, its category of integrating Lie groupoids has similar properties to those described above, but with some important differences. For instance, there is the question of which integrations are Hausdorff.  Also, unlike the case of Lie groups, where morphisms among integrations are covering maps, for Lie groupoids these morphisms are only local diffeomorphisms, which may fail to be covering maps. 

We now define the category of integrations of a Lie algebroid; for this we need the groupoid analog of (simple-)connectedness.
\begin{definition} \label{Defn: SscGpd} 
Any Lie groupoid $\Gg \rightrightarrows M$ has a well-defined subgroupoid $\Gg^c\rightrightarrows M$ all of whose source fibres are connected.  If $\Gg=\Gg^c$, we say that $\Gg$ is \defw{source-connected}.

If the source fibres of a source-connected groupoid are also simply connected, then the groupoid is called \defw{source-simply-connected}, or \defw{ssc} for short.    
\end{definition}

By a result of Moerdijk-Mr\v{c}un~\cite{MR2012261}, if a Lie algebroid $A$ is integrable, then it has a source-simply-connected integration $\Gg^{ssc}$, unique up to a canonical isomorphism.  Their results also show that any source-connected integration $\Gg$ of $A$ receives a unique morphism
\begin{equation}\label{surjgrp}
p:\Gg^{ssc}\to \Gg,
\end{equation}
which is a surjective local diffeomorphism.  As a result, $\Gg^{ssc}$ may be viewed as the initial object of a category of integrations.

\begin{definition}
To a Lie algebroid $A$ over $M$, we associate two categories 
\[
\Lint^\Hh(A)\subset \Lint(A):
\]
\begin{enumerate}
\item Objects of $\Lint(A)$ are pairs $(\Gg, \phi)$, where $\Gg$ is a source-connected Lie groupoid over $M$ and $\phi: \Lie(\Gg) \rightarrow A$ is an isomorphism covering the identity on $M$; 
\item A morphism from $(\Gg, \phi)$ to $(\Gg', \phi')$ is a Lie groupoid morphism $\psi: \Gg \rightarrow \Gg'$ such that $\phi = \phi' \circ \Lie(\psi)$.   
\end{enumerate}
The subcategory of Hausdorff integrations is then denoted by $\Lint^\Hh(A)$.
When $\Lint(A)$ is nonempty, there is an initial object, called the \emph{ssc} groupoid $\Gg^{ssc}$.  A terminal object, when it exists, is called the \emph{adjoint} groupoid $\Gg^{adj}$.   It follows from~\eqref{surjgrp} that there is at most one morphism between any two objects, and that this morphism is a surjective local diffeomorphism.
\end{definition}

In analogy with the theory of covering spaces, each of the integrations in $\Lint(A)$ may be described as a quotient of the source-simply-connected integration.  This defines an equivalence between integrations and normal subgroupoids of $\Gg^{ssc}$.    

\begin{theorem} \label{Thm: Gpd-NormalSubgpd}
A ssc integration $\Gg^{ssc}$ of the Lie algebroid $A$ defines an equivalence of categories between $\Lint(A)$ and the poset $\Lnorm(\Gg^{ssc})$ of discrete, totally disconnected, normal Lie subgroupoids of $\Gg^{ssc}$, via 
\[
\mathbf{N}: \Lint(A) \rightarrow \Lnorm(\Gg^{ssc}),
\]
which takes a groupoid $\Gg$ to the kernel of the canonical morphism $p:\Gg^{ssc}\to \Gg$, and
\[
\mathbf{G}: \Lnorm(\Gg^{ssc}) \rightarrow \Lint(A),
\]
which takes the normal subgroupoid $\Nn\vartriangleleft\Gg^{ssc}$ to the quotient groupoid $\Gg^{ssc} / \Nn$.  

In the case that $\Gg^{ssc}$ is itself Hausdorff, then the equivalence identifies $\Gpd^\Hh(A)$ with the subposet $\Lnorm^\Hh(A)\subset \Lnorm(A)$ of closed subgroupoids.
\end{theorem}

\begin{proof}
Since the morphism $p: \Gg^{ssc} \rightarrow \Gg$ is base-preserving, the normal subgroupoid $\mathbf{N}(\Gg) = \ker(p)$ is totally disconnected. Since $\Lie(\Gg^{ssc}) = \Lie(\Gg)$, the subgroupoid $\mathbf{N}(\Gg)$ is discrete.  If $\Gg^{ssc}$ is Hausdorff, then by Theorem~\ref{Prop: LieGpdQuotient}, $\mathbf{N}(\Gg)$ is closed.

Conversely, by Theorem~\ref{Prop: LieGpdQuotient}, since $\Nn$ is a normal, totally disconnected Lie subgroupoid, $\mathbf{G}(\Nn) = \Gg^{ssc} / \Nn$ is a Lie groupoid over the same base as $\Gg^{ssc}$. Since $\Nn$ is discrete, it follows that $\Gg^{ssc} / \Nn$ integrates $A$.  Finally, if $\Gg^{ssc}$ is Hausdorff, then $\Gg^{ssc}/\Nn$ is Hausdorff iff $\Nn$ is closed.  
\end{proof}

%
%
%

For the Lie algebroids $T_DM$, $T^*_\pi M$ under consideration in this paper, the category of integrations described above is somewhat simplified, in that for any source-connected integration $(\Gg,\phi)$ of $A$, the isomorphism $\phi:\Lie(\Gg)\to A$ is uniquely determined by the groupoid $\Gg$, and may be ignored. This derives from the fact that the Lie algebroids are~\emph{almost injective}~\cite{MR1881646}:
\begin{definition}
A Lie algebroid is called \emph{almost injective} when its anchor map $a:A\to TM$ induces an injection of sheaves of local sections.  In other words, the anchor map is an injective bundle map over an open dense subset of $M$.
\end{definition}
A Lie algebroid automorphism $\phi:A\to A$ covering $\id_M$ must commute with the anchor map, i.e. $a\circ \phi = \phi$, and so we immediately obtain the following rigidity result:
\begin{prop} \label{Proposition: SubsheafLieAlgbroid}
If $A$ is an almost injective Lie algebroid over $M$, then the only automorphism of $A$ covering $\id_M$ is the identity map.
\end{prop}


\begin{example} \label{ex: tanggpd}
The tangent Lie algebroid $TM$ of a connected manifold $M$ has ssc integration given by the fundamental groupoid $\Pi_1(M)$, and has adjoint groupoid given by the pair groupoid $\Pair(M)=M\times M$.  We determine all integrations of $TM$ as follows.  By Theorem~\ref{Thm: Gpd-NormalSubgpd}, any integration $\Gg$ of $TM$ may be described as a quotient of $\Pi_1(M)$ by a discrete, totally disconnected normal Lie subgroupoid $\Nn$.  

Since $\Nn$ is totally disconnected, it is contained in the isotropy subgroupoid of $\Pi_1(M)$, which is discrete, so that $\Nn$ is automatically discrete.  The fact that $\Nn$ is normal implies that $\Nn$ is determined by its intersection with the isotropy group at any point $x_0\in M$, which is simply the fundamental group $\Pi_1(x_0,x_0) = \pi_1(M, x_0)$ based at $x_0$.  Hence $\Nn$ is uniquely determined by the choice of a normal subgroup of the fundamental group of $M$, and so the category of integrations is equivalent to the lattice of normal subgroups of the fundamental group of $M$:  
\[
\Lint^\Hh(TM) \cong \Lint(TM) \cong \Lnorm(\pi_1(M)).
\]
The integrations of $TM$ are all Hausdorff, since the normal subgroupoids described above are closed in $\Pi_1(M)$.
\end{example}

Given a Poisson manifold $(M, \pi)$, the ssc integration $\Gg^{ssc}$ of the Lie algebroid $T^*_\pi M$ inherits a natural multiplicative symplectic structure $\Omega$, making $(\Gg^{ssc}, \Omega)$ a symplectic groupoid for $(M, \pi)$. In general, however, other integrations of $T^*_\pi M$ need not admit multiplicative symplectic structures.  On the other hand, multiplicative symplectic forms behave well under pullbacks, in the following sense.
\begin{prop} \label{prop: multisymp}
Let $\phi: \Gg' \rightarrow \Gg$ be a morphism between groupoids integrating $T^*_\pi M$. If $\Omega\in\Omega^2(\Gg)$ is multiplicative and symplectic, then so is $\phi^*\Omega\in\Omega^2(\Gg')$. 
\end{prop}
%
%
In the following section, we show that for proper log symplectic manifolds, the adjoint integration has a natural multiplicative symplectic form, i.e. is a symplectic groupoid. By Proposition~\ref{prop: multisymp}, all other integrations are symplectic for this class of Poisson manifolds.

\section{Birational construction of adjoint groupoids} \label{Section: Blowup}
In this section, we systematically develop a blow-up operation in the category of Poisson groupoids, and use this operation to contruct the adjoint symplectic groupoid of a proper log symplectic manifold. This groupoid is a special case of the class of groupoids constructed by Debord in~\cite{MR1881646}, but the construction is different, inspired by and extending the  
work of Melrose~\cite{MR1348401} and Monthubert~\cite{MR1600121}.  We hope that the explicit and global nature of the construction will make the (adjoint) symplectic groupoid more accessible from a geometric point of view. 

	\subsection{Real projective blow-up}\label{blowp}

	Let $M$ be a real smooth manifold, with a closed submanifold $L\subset M$, such that
\[
\cod(L) \geq 2.
\]
We denote by $\blowup{M}{L}$ the real projective blow-up of $M$ along $L$. Recall that to construct $\blowup{M}{L}$ from $M$, we replace $L$ by the projectivisation of its normal bundle, $\mathbb{P}(NL)$, which then defines a hypersurface $E\subset \blowup{M}{L}$ called the exceptional divisor. The blow-down map
\[
p: \blowup{M}{L} \rightarrow M
\]
is a diffeomorphism away from $E$ and coincides with the bundle projection $\mathbb{P}(NL) \rightarrow L$ upon restriction to the exceptional divisor. 

Any submanifold $S\subset M$ having clean\footnote{Submanifolds $S, L$ have clean intersection when $S\cap L$ is a submanifold and $T(S\cap L) = TS\cap TL$.} intersection with $L$ may be ``pulled back'' to $\blowup{M}{L}$, by forming the \defw{proper transform} (a.k.a the strict transform)
\[
\thickbar{S} \coloneqq\overline{p^{-1}(S \backslash L)},
\]
where the closure is taken in $\blowup{M}{L}$.  The proper transform $\thickbar{S}$ is itself a submanifold, naturally isomorphic to $\blowup{S}{L\cap S}$.  Of course, if $L\cap S$ has codimension 1 in $S$, the blowdown map restricts to a diffeomorphism $\thickbar{S}\to S$. 

The projective blow-up $\blowup{M}{L}$ is characterized by a universal property~(\cite{MR0463157}, Prop. II.7.14): a smooth map $f:X\to M$ has a unique lift $\widetilde{f}$ to the blowup if and only if the pullback $f^*\mathcal{I}_L$ of the ideal sheaf of $L$ defines a line bundle.  
\begin{equation} \label{Diagram: Universal-BlowUp}
\begin{aligned}
\xymatrix{ & \blowup{M}{L} \ar[d]^-{p} \\
X \ar[r]_-{f} \ar@{.>}[ur]^-{\widetilde{f}}	& M }
\end{aligned}
\end{equation}
A special case of this result, more useful for our purposes, is the following.

\begin{prop} \label{univblowup2}
Let $L\subset M$ be a closed submanifold of codimension $\geq 2$.  If $f:X\to M$ is a smooth map and $Y=f^{-1}(L)\subset X$ is a hypersurface such that the bundle map 
$$Nf: NY\to f^*NL$$ 
induced by the derivative $Tf$ is injective, then there exists a unique smooth map $\widetilde{f}:X\to\blowup{M}{L}$ such that $f=p\circ \widetilde{f}$.
\end{prop}

\begin{proof}
If $Nf: NY \rightarrow f^*NL$ is injective, then the dual map
\[
f^*(\mathcal{I}_L / \mathcal{I}_L^2) \rightarrow \mathcal{I}_Y / \mathcal{I}_Y^2
\]
is surjective (here $\mathcal{I}_L$, $\mathcal{I}_Y$ denote the ideal sheaves of $L, Y$ respectively), implying that $f^*(\mathcal{I}_L) = \mathcal{I}_Y$.  But $\mathcal{I}_Y$ is a line bundle when $Y$ is a hypersurface, so by the universal property of the projective blow-up, the result holds.
\end{proof}
\begin{remark}\label{bldninj}
The blow-down map $p:\blowup{M}{L}\to M$ itself satisfies Proposition~\ref{univblowup2}, since $Np:N(p^{-1}(L))\to p^*NL$ is injective; the resulting lift is simply the identity map on $\blowup{M}{L}$.
\end{remark}






	\subsection{Blow-up of Poisson manifolds} \label{Section: Blow-upPoisson}

	Let $(M,\pi)$ be a Poisson manifold, and let $L\subset M$ be a Poisson submanifold.  The normal space $N_p L$ to any point $p\in L$ then inherits a linear Poisson structure, defining a transverse Poisson structure
\begin{equation}\label{tranpois}
\pi_N\in\Gamma(L, N^*L\otimes\wedge^2 NL),
\end{equation}
which exhibits the conormal bundle $N^*L$ as a bundle of Lie algebras.

In~\cite{MR1465521}, Polishchuk observed that in order for the Poisson structure on $M$ to lift to the blow-up $\blowup{M}{L}$, the transverse Poisson structure along $L$ must be \defw{degenerate}, in the following sense.

\begin{definition} \label{degliealg}
The transverse Poisson structure $\pi_N$ of a Poisson submanifold $(L,\pi|_L)\subset (M,\pi)$ is called \emph{degenerate} when 
\[
\pi_N = v\wedge E,
\]
where $v$ is the normal vector field obtained from $\pi_N$ via the contraction $N^*{L}\otimes \wedge^2 NL\to NL$, and $E$ is the Euler vector field on $NL$. 
\end{definition}
\begin{remark}
The Lie algebra on each conormal space $N_p^*L$ induced by a degenerate transverse Poisson structure is either abelian, when $v(p)=0$, or isomorphic to the semidirect product Lie algebra $\RR\ltimes \RR^{n-1}$ associated to the action $a\cdot u = au$ of $\RR$ on $\RR^{n-1}$.
\end{remark}

\begin{theorem}[Polishchuk \cite{MR1465521}] \label{poissonblowup}
Let $(L, \pi|_{L})\subset (M, \pi)$ be a closed Poisson submanifold with degenerate transverse Poisson structure $\pi_N$.  Then there is a unique Poisson structure $\til{\pi}$ on $\blowup{M}{L}$ such that $p_*(\til{\pi}) = \pi$. Furthermore, the exceptional divisor is Poisson if and only if $\pi_N$ vanishes.
\end{theorem}

We will apply the above result only in two special cases, which we detail below. 

\begin{prop} \label{bbsympblowup}
Let $(M_1,\pi_1)$ and $(M_2,\pi_2)$ be log symplectic manifolds with degeneracy loci $D_1$ and $D_2$, respectively. Then the natural Poisson structure $\til{\pi}$ on the blowup $\blowup{M_1 \times M_2}{D_1 \times D_2}$ is log symplectic on  
\[
\widetilde{X} \coloneqq \blowup{M_1 \times M_2}{D_1 \times D_2} \backslash (\thickbar{M_1 \times D_2} \cup \thickbar{D_1 \times M_2}),
\]
with degeneracy locus given by the exceptional divisor $\widetilde{X} \cap \PP (NL)$ in $\til{X}$.
\end{prop}

\begin{proof}
Let $M = M_1\times M_2$, and note that the transverse Poisson structure $\pi_N$ along $L = D_1\times D_2$ vanishes. By Theorem \ref{poissonblowup}, there is a unique Poisson structure $\til{\pi}$ on $\blowup{M}{L}$ such that $p_*(\til{\pi}) = \pi_1\oplus \pi_2$. We must show that the Pfaffian of $\til{\pi}$ vanishes transversely along $\blowup{M}{L} \cap \PP(NL)$.

By the Linearization theorem~\ref{Thm: GMP}, we may, as in Proposition~\ref{totpois}, take $M_i = \tot(N_i)$ to be the total space of a line bundle $N_i$ over $D_i$, and after choosing connections we obtain 
\[
\pi_i = \sigma_i + Z_i \wedge E_i,
\]
where $Z_i$ and $\sigma_i$ are vector and bivector fields on $D$, respectively, and $E_i$ is the Euler vector field on $\tot(N_i)$, which may be viewed as a section $E_i\in\Gamma(D_i,N^*D_i\otimes ND_i)$.

Then, if $\dim M_i = 2n_i$, the Pfaffian of $\pi = \pi_1\oplus \pi_2$ is 
\begin{equation}\label{factorizz}\begin{aligned}
\pi^{n_1 + n_2} &= c  E_1\wedge E_2\wedge Z_1\wedge\sigma_1^{n_1-1}\wedge Z_2\wedge\sigma_2^{n_2-1}\\
&= c  E_1\wedge E_2\wedge\chi_1\wedge\chi_2,
\end{aligned}\end{equation}
where $c$ is a nonzero constant and $\chi_i\in\Gamma(D_i,\wedge^{n_i}TD_i)$ are the residues from Definition~\ref{residef}.  Hence we see that the Pfaffian may be viewed invariantly as a nonvanishing section
\[
\pi^{n_1+n_2} \in \Gamma(L, \Sym^2 N^*L\otimes \wedge^2 NL \otimes\det TL).
\]
After blowing up, this section defines a covolume form which vanishes linearly on the exceptional divisor, in the following way.  First, the blowup $\blowup{M}{L}$ may be identified with the total space of the tautological line bundle $q:U\to\PP(NL)$.  Covolume forms on $\tot(U)$ which vary linearly on the fibers are sections of 
	\begin{equation}\label{ein}
		q^*U^*\otimes \det (T(\tot(U))) = q^*\det T\PP(NL).
	\end{equation}
Using the blow-down map $p:\PP(NL)\to L$, we write the Euler sequence 
	\begin{equation}\label{zwei}
		\xymatrix{0\ar[r] & \underline{\RR}\ar[r] & U^*\otimes p^* NL\ar[r] & V\ar[r] & 0},
	\end{equation}
defining the relative tangent bundle $V$ for the projection $p$. Also, we have the exact sequence
	\begin{equation}\label{drei}
		\xymatrix{0\ar[r] & V\ar[r] & T\PP(NL)\ar[r] & p^*TL\ar[r] & 0}.
	\end{equation}
Combining~\eqref{ein}, \eqref{zwei}, and \eqref{drei}, we see that fibrewise linear covolume forms on $\tot(U)$ are given by sections of
\[
\det(U^*\otimes p^*NL)\otimes\det p^*TL = (U^*)^2\otimes p^*(\wedge^2 NL\otimes \det TL),
\]
where we have used the fact that $L$ is codimension 2.

Squaring the restriction $p^*N^*\to U^*$, we obtain a natural map $r: p^*\Sym^2 N^*L \to (U^*)^2$, so that the Pfaffian defines a section 
\[
r\otimes(\pi^{n_1+n_2}) \in \Gamma(\PP(NL), (U^*)^2\otimes p^*(\wedge^2 NL\otimes \det TL)).
\]
Therefore, after blow-up, the Pfaffian defines a fibrewise linear covolume form on $\tot(U)$ which varies quadratically along the projective fibres, vanishing (due to the factorization~\eqref{factorizz}) along the fibres over the pair of sections $\PP(N(M_1\times D_2))$, $\PP(N(D_1\times M_2))$ of $\PP(NL)$, which coincide with the loci $\thickbar{M_1 \times D_2}$ and $\thickbar{D_1 \times M_2}$ along the exceptional divisor, as required.
\end{proof}

\begin{prop} \label{bsympblowup}
Let $(M, \pi)$ be a log symplectic manifold with degeneracy locus $D$. If $F \subset D$ is a symplectic leaf, then there is a unique log symplectic Poisson structure $\til{\pi}$ on $\blowup{M}{F}$ such that $p_*(\til{\pi}) = \pi$. Moreover, $\til{\pi}$ is non-degenerate (i.e., symplectic) on $\blowup{M}{F} \backslash \thickbar{D}$.
\end{prop}

\begin{proof}
As in Proposition~\ref{bbsympblowup}, we may assume, by linearization and upon choosing flat connections, that $M = \tot(N_M F)$ and that $\pi$ has the form 
\[
\pi = \sigma + Z \wedge E,
\]
where $\sigma\in\Gamma(F,\wedge^2TF)$ is a nondegenerate Poisson structure, $E$ is the Euler vector field on $N_MF$, and $Z\in\Gamma(F, N_DF)$ is nonvanishing.  In this case, the transverse Poisson structure is $\pi_N = Z \wedge E$, hence degenerate, so by Theorem \ref{poissonblowup}, there is a unique lift $\til{\pi}$ of the Poisson structure to $\blowup{M}{F}$. We need to show that the Pfaffian of $\til{\pi}$ vanishes only, and transversely, along $\widetilde{X} \cap \PP K$.

If $\dim M = 2n$, then the Pfaffian $\pi^n$ is a nonzero multiple of the section
\begin{equation} \label{eq: pfaffian-linear}
\begin{aligned} 
E \wedge Z \wedge \sigma^{n-1} = E \wedge\chi, 
\end{aligned}
\end{equation}
where $\chi$ is the residue of $N_M D|_F$ as a Poisson line bundle over $D$. This means that the Pfaffian defines a nonvanishing section of 
\[
\pi^n\in \Gamma(F, N_D F \otimes \det(TF)).
\]
To compute the blowup of $\pi^n$, we use the same method as in Proposition~\ref{bbsympblowup}.  We identify the blowup $\blowup{M}{F}$ with the total space of the tautological bundle $q:U\to \PP(N_M F)$, and then we note that the space of fibrewise constant covolumes on $\tot(U)$ is given by 
\[
\Gamma(\PP(N_M F), U^*\otimes q^*(\wedge^2 N_MF\otimes \det TF)).
\]
Composing the inclusion $U\to q^*N_M F$ with the projection in the exact sequence 
\[
\xymatrix{0\ar[r]& N_D F\ar[r] & N_M F\ar[r] & N_M D|_F\ar[r] & 0},
\]
we obtain a map $j:U\to N_MD|_F$ which vanishes precisely along the section $\PP(N_D F)\subset \PP(N_MF)$, and transversely.  So, on the blowup, the Pfaffian is given by 
\[
j\otimes \pi^n,
\]
which is nonzero along the exceptional divisor, except for a transversal zero along $\PP(N_D F)$, which is the intersection of $\thickbar{D}$ with the exceptional divisor.

\end{proof}

%

	\subsection{Blow-up of Lie groupoids} \label{Section: Blow-upGpd}

	In this section, we demonstrate that the projective blow-up of a Lie groupoid $\Gg \rightrightarrows M$ along a subgroupoid
$\Hh \rightrightarrows L$ inherits a Lie groupoid structure, once a certain degeneracy locus is removed.  We restrict our attention to the case that the base $L$ of the subgroupoid has codimension 1.

\subsubsection{Lifting theorem}

To lift the groupoid operations from $\Gg$ to the blow-up, we utilize the lifting criterion for smooth maps given in Proposition~\ref{univblowup2}.  The key point is the lifting of the groupoid multiplication; to apply the criterion here, we need the following description of the normal bundle of a fibre product of smooth maps of pairs.  Recall that if $X\subset Y$ and $L\subset M$ are submanifolds, then $f:(Y,X)\to (M,L)$ is a smooth map of pairs when $f:Y\to M$ is a smooth map such that $f(X)\subset L$.  Such a morphism of pairs induces a morphism of normal bundles 
\[
Nf: NX\to f^*NL,
\] 
defined to be the quotient of $(Tf)|_{X}:TY|_X\to f^* TM$ by $T(f|_X):TX\to f^*TL$.
\begin{lemma}\label{Lemma: MultiplicationSmooth}
	Let $f_1:(Y_1,X_1)\to (M,L)$ and $f_2: (Y_2,X_2)\to (M,L)$ be transverse smooth maps of pairs\footnote{That is, both $f_1:Y_1\to M$, $f_2:Y_2\to M$ and $f_1|_{X_1}:X_1\to L$, $f_2|_{X_2}:X_2\to L$ are transverse.}.  Then the fiber product of submanifolds $X_1\times_L X_2\subset Y_1\times_M Y_2$ has normal bundle given by 
	\[
		N(X_1 \times_L X_2) = NX_1 \times_{NL} NX_2.
	\]
\end{lemma}


\begin{theorem} \label{liegpdblowup} Let $\Hh \rightrightarrows L$ be a closed Lie subgroupoid of $\Gg \rightrightarrows M$ over the closed hypersurface $L$, and define
    \begin{equation}\label{defngpdblowup}
    \BGpd{\Gg}{\Hh} \coloneqq \blowup{\Gg}{\Hh} \backslash
    (\thickbar{s^{-1}(L)} \cup \thickbar{t^{-1}(L)}),
    \end{equation}
    where $s$ and $t$ are the source and target maps of $\Gg$.
    There is a unique Lie groupoid structure
    $\BGpd{\Gg}{\Hh} \rightrightarrows M$ such that the blow-down map restricts to a base-preserving Lie groupoid morphism 
    \[
    p: \BGpd{\Gg}{\Hh} \rightarrow \Gg.
    \]
    The exceptional locus $\til{\Hh} = \BGpd{\Gg}{\Hh}\cap p^{-1}(\Hh)$ is then a Lie subgroupoid of codimension 1 in $\BGpd{\Gg}{\Hh}$, along which $p$ restricts to a morphism $\til{\Hh}\to \Hh$ of Lie groupoids over $L$.
\end{theorem}
\begin{proof} 
    To obtain the result, we lift the groupoid structure on $\Gg$ to maps on $\widetilde{\Gg}=\BGpd{\Gg}{\Hh}$ and verify that the groupoid axioms are satisfied.  First, if the blow-down $p: \widetilde{\Gg} \rightarrow \Gg$ is to be a base-preserving Lie groupoid morphism, the source and target maps of $\widetilde{\Gg}$ must be given by 
    \[
    \widetilde{s} = s \circ p,\qquad \widetilde{t} = t \circ p.
    \]
    In the following steps, we show that $\widetilde{s}, \widetilde{t}$ are submersions, and then obtain the remaining lifts: the multiplication
    $m: \Gg^{(2)}\rightarrow \Gg$ lifts uniquely to
    $\widetilde{m}: \widetilde{\Gg}\times_{M}\widetilde{\Gg} \rightarrow \widetilde{\Gg}$; the
    identity $\id: \Gg \rightarrow \Gg$ lifts uniquely to
    $\widetilde{\id}: M \rightarrow \widetilde{\Gg}$; and the inverse map
    $i: \Gg \rightarrow \Gg$ lifts uniquely to
    $\widetilde{i}: \widetilde{\Gg} \rightarrow \widetilde{\Gg}$.

  Once the lifts are defined, we see that they satisfy the groupoid conditions on an open dense set (the complement of $\til{\Hh}$); by continuity, the groupoid axioms hold on all of $\widetilde{\Gg}$ and $p: \widetilde{\Gg} \rightarrow \Gg$ is a Lie groupoid morphism, completing the proof.

    \vspace{1ex}
    \noindent \textbf{Step 1:} {\it The maps $\widetilde{s}, \widetilde{t}$ are submersions of pairs $(\widetilde{\Gg},\widetilde{\Hh})\to (M, L)$.}
	\vspace{1ex}
	
    The blow-down $p: \widetilde{\Gg} \rightarrow \Gg$ is a local
    diffeomorphism away from $\widetilde{\Hh}$, so it suffices to show that
    $T\widetilde{s}: (T\widetilde{\Gg}|_{\widetilde{\Hh}}, T\widetilde{\Hh}) \rightarrow (TM|_L, TL)$ 
    is pairwise surjective. Since $Tp: T\widetilde{\Hh} \rightarrow T\Hh$ and $Ts:T\Hh\to TL$ are 
    surjective, we immediately obtain that 
    $T\widetilde{s}|_{\widetilde{\Hh}}: T\widetilde{\Hh} \rightarrow TL$  is surjective.
    \begin{equation} \label{Diagram: Source-NormalSurj}
        \begin{aligned}
            \xymatrix{
                0 \ar[r]	& T\widetilde{\Hh} \ar[d]^-{T\widetilde{s}|_{\widetilde{\Hh}}} \ar[r]	& T\widetilde{\Gg}|_\Hh \ar[d]^-{T\widetilde{s}} \ar[r]	& N\widetilde{\Hh} \ar[d]^-{N\widetilde{s}} \ar[r]	& 0 \\
                0 \ar[r] & TL \ar[r] & TM \ar[r] & NL \ar[r] & 0 }
        \end{aligned}
    \end{equation}
    By the commutative diagram~\eqref{Diagram: Source-NormalSurj}, it remains to show that the normal map
    $N\widetilde{s}: N\widetilde{\Hh} \rightarrow NL$ induced by $T\widetilde{s}$ is surjective.
%
    Now, since $s:(\Gg,\Hh)\to (M,L)$ is a submersion, the induced map $Ns:N\Hh\to NL$ is surjective, 
    with kernel subbundle $K_s\subset N\Hh$, and the composition $N\widetilde{s} = Ns\circ Np$ therefore fails to be surjective 
    along the subset $\PP(K_s)\subset\PP(N\Hh)=p^{-1}(\Hh)$ of the exceptional divisor.  But $\PP(K_s) = p^{-1}(\Hh)\cap \thickbar{s^{-1} L}$ has been removed in the definition~\eqref{defngpdblowup}, so that $N\widetilde{s}$ is surjective along $\Hh$.  The same argument applies to the target map $t$, yielding the result.

    \vspace{1ex}
    \noindent \textbf{Step 2:} {\it Lifting the multiplication map.}
    \vspace{1ex}
    
	To lift the multiplication $m: \Gg^{(2)} \rightarrow \Gg$ to a map $\widetilde m: \widetilde{\Gg}^{(2)} = \widetilde{\Gg}\times_{M}\widetilde{\Gg}\to \widetilde{\Gg}$, we apply the universal property of blow-up to the map $f = m\circ(p\times p)$ in order to complete the following commutative diagram.
    \begin{equation*}
    \begin{aligned}\xymatrix{
            \widetilde{\Gg}^{(2)} \ar[d]_-{p \times p} \ar@{.>}[r]^-{\widetilde{m}} \ar[dr]_-{f}		&	\blowup{\Gg}{\Hh} \ar[d]^-{p}	\\
            \Gg^{(2)} \ar[r]_-{m} & \Gg }
    \end{aligned}
    \end{equation*}
	By Proposition \ref{univblowup2}, it suffices to show that $f^{-1}(\Hh)$ is a
    hypersurface in $\widetilde{\Gg}^{(2)}$ whose normal bundle injects into $N\Hh$ via $Nf$.
	First, note that $f^{-1}(\Hh) = \widetilde{\Hh}^{(2)} \coloneqq \widetilde{\Hh}\times_L \widetilde{\Hh}$, 
	which is a hypersurface (smooth and codimension 1) by transversality.  
	Then, to show that $Nf: N\widetilde{\Hh}^{(2)} \rightarrow N\Hh$ is
    injective, we show the stronger result that $s\circ f$ induces an isomorphism $N\widetilde{\Hh}^{(2)}\to NL$. Observe that 
    \[
    s\circ f = s\circ m\circ(p\times p) = s\circ p_1\circ (p\times p)=\til{s}\circ\til{p_1},
    \]
    where $p_1$ and $\til{p_1}$ are the first projections ${\Gg}\times_M{\Gg}\to {\Gg}$ 
    and $\til{\Gg}\times_M\til{\Gg}\to \til{\Gg}$, respectively.  
  	The induced map on normal bundles is then the following composition
	\[
	\xymatrix{N(\til{\Hh}\times_L\til{\Hh}) = N\til{\Hh}\times_{NL}N\til{\Hh}\ar[r]^-{N\til{p_1}} &N\til{\Hh}\ar[r]^-{N\til{s}} &NL },	
	\]  
    where we have used Lemma~\ref{Lemma: MultiplicationSmooth} to compute $N\til{\Hh}^{(2)}$.  The composition is an isomorphism since $N\til{s}$ is surjective by Step 1 and $N\til{p_1}$ is surjective between bundles of rank 1.
    
    Since $N(s\circ f)$ is an isomorphism, it follows that $Nf(N\widetilde{\Hh}^{(2)}) \cap K_s = 0$ and $\til{m}(\widetilde{\Gg}^{(2)}) \cap \PP (K_s) = \varnothing$. Similarly $\til{m}(\widetilde{\Gg}^{(2)}) \cap P(K_t) = \varnothing$ and we obtain $\til{m}(\widetilde{\Gg}^{(2)}) \subset \til{\Gg}$.

	\vspace{1ex}
	\noindent \textbf{Step 3:} {\it Lifting the identity map.}
	\vspace{1ex}

	The identity $\id: (M,L) \rightarrow (\Gg,\Hh)$ is an embedding of pairs such that $\Hh\cap \id(M) = \id(L)$.  This implies that $T\Hh\cap T(\id(M)) = T(\id(L))$, so that the induced map on normal bundles $N(\id):NL\to N\Hh$ is injective.  Since $\id^{-1}(\Hh) = L$ is a hypersurface by assumption, Proposition~\ref{univblowup2} yields the unique lift $\til{\id}: M \rightarrow \blowup{G}{H}$.
	
	Since $T(\id(M))$ is complementary to $T(s^{-1}(x))$ at $\id(x) \in \id(M) \subset \Gg$, it follows that $N(\id)(NL) \cap K_s = 0$ and $\til{\id}(M) \cap \PP (K_s) = \varnothing$. Similarly, $\til{\id}(M) \cap \PP (K_t) = \varnothing$, and it follows that $\til{\id}(M) \subset \til{\Gg}$.
	
	\vspace{1ex}
    \noindent \textbf{Step 4:} {\it Existence of inverses.}
    \vspace{1ex}

	To obtain an inverse map $\til{i}$ for $\til{\Gg}$, we lift $i\circ p$, completing the commutative diagram
	\[
    \xymatrix{
        \blowup{\Gg}{\Hh} \ar[d]_-{p} \ar@{.>}[r]^-{\widetilde{i}}		&	\blowup{\Gg}{\Hh} \ar[d]^-{p}	\\
        \Gg \ar[r]_-{i} & \Gg }
    \]
	First, note that $(i\circ p)^{-1}(\Hh)$ is the exceptional divisor of $\blowup{\Gg}{\Hh}$, so it remains to show that $N(i\circ p) = Ni\circ Np$ is injective.  Since the inverse map $i:(\Gg,\Hh)\to (\Gg,\Hh)$ is a diffeomorphism of pairs, $Ni$ is an isomorphism.  For any blow-down, $Np$ is injective along the exceptional divisor, as in Remark~\ref{bldninj}. By Proposition \ref{univblowup2}, we obtain the unique lift $\til{i}: \blowup{\Gg}{\Hh} \rightarrow\blowup{\Gg}{\Hh}$. since the inversion on $G$ exchanges $s^{-1}(L)$ and $t^{-1}(L)$, it follows that $\til{i}$ exchanges $\thickbar{s^{-1}(L)}$ and $\thickbar{t^{-1}(L)}$, so that $\til{i}(\til{G}) \subset \til{G}$, as required.
\end{proof}

\subsubsection{Elementary modification of Lie algebroids}

The blow-up operation for Lie groupoids given in Theorem~\ref{liegpdblowup} 
corresponds to an operation on Lie algebroids, which we now describe.  We begin by simply applying the Lie functor as in~\S\ref{liefunctor} to the blow-up groupoid 
$\BGpd{\Gg}{\Hh}$.

\begin{corollary} \label{c: blow-up agd} Let $\Gg \rightrightarrows M$ be a Lie
    groupoid, and $\Hh \rightrightarrows L$ a closed Lie subgroupoid over the closed hypersurface $L$, so that $\Lie(\Hh)$ is a Lie subalgebroid\footnote{The 
    standard notion of Lie subalgebroid is described in~\cite[Definition 4.3.14]{MR2157566}.}
    of $\Lie(\Gg)$.  
    Then $\Lie(\BGpd{\Gg}{\Hh})$ has sheaf of sections 
	defined by
    \begin{equation}\label{lieofab}
	\Lie(\BGpd{\Gg}{\Hh}) = \{X \in \Lie(\Gg) ~|~  X|_L \in \Lie(\Hh) \}.
    \end{equation}
\end{corollary}
\begin{proof}
    For any Lie groupoid, we may view the sections of its Lie algebroid as
    left-invariant vector fields (always taken to be tangent to the source fibres).  
    Therefore, it suffices to show that the blow-down map $p:\BGpd{\Gg}{\Hh}\to \Gg$ induces
    a bijection between the right hand side of~\eqref{lieofab}, viewed as left invariant vector fields on $\Gg$ tangent to $\Hh$, and the left hand side of~\eqref{lieofab}, viewed as left invariant vector fields on $\BGpd{\Gg}{\Hh}$.  
    
    Let $X$ be a left-invariant vector field on $\Gg$ tangent to $\Hh$, 
    and let $\phi: I \times \Gg \rightarrow \Gg$ 
    be its flow, defined on a sufficiently small neighbourhood 
    $I\subset\RR$ of zero.  We show that $\phi$ lifts to a flow on $\BGpd{\Gg}{\Hh}$ by 
    first lifting the map to the blow-up $\blowup{I\times\Gg}{I\times\Hh}$, which completes 
    the commutative diagram (here $p, p'$ are the blow-down maps)
    \begin{equation} \label{Diagram: GpdLift} 
    \begin{aligned}
    \xymatrix{
	\blowup{I \times\Gg}{I \times\Hh} \ar[d]_-{p'} \ar@{.>}[r]^-{\widetilde{\phi}}& \BGpd{\Gg}{\Hh} \ar[d]_-{p}	\\
                I \times\Gg \ar[r]_-{\phi} &
            \Gg } 
    \end{aligned}\end{equation}
    and then noting that 
    $\blowup{I \times\Gg}{I \times\Hh} = I \times \blowup{\Gg}{\Hh}$, so that 
    $\til{\phi}$ is indeed a flow on $\BGpd{\Gg}{\Hh}$.  
    Then $\til{X} = d\til{\phi}/dt|_{t=0}$ is the required lift 
    of $X$ to a left invariant vector field on $\BGpd{\Gg}{\Hh}$.
    The lift is obtained via Proposition~\ref{univblowup2}, as follows: since $X$ is tangent to $\Hh$, we have
    $\phi^{-1}(\Hh) = I \times \Hh$, and so $(\phi\circ p')^{-1}(\Hh)$ is the
    exceptional divisor in $\blowup{I \times\Gg}{I \times\Hh}$, a hypersurface. 
    Furthermore, $N(\phi\circ p')$ is the composition of $Np'$, injective by Remark~\ref{bldninj}, with $N\phi$, an isomorphism, so is itself injective, proving existence and uniqueness of $\til{\phi}$.
    
    Conversely, we show that $\Lie(\BGpd{\Gg}{\Hh})$ is generated by 
    the lifts of left-invariant vector fields obtained above.      
    For a sufficiently small neighbourhood $U\subset M$ of $p\in L$, 
    choose a basis of sections of $\Lie(\Hh)$ over $U\cap L$ and 
    extend them to linearly independent sections
    $(X_1,\ldots, X_l)$ of $\Lie(\Gg)$ over $U$.  
    Extend this to a basis $(X_1, \ldots, X_l, X_{l+1},\ldots, X_n)$ of $\Lie(\Gg)$ over $U$.
    Then, if $f\in C^\infty(U,\RR)$ is a generator for the ideal sheaf of $L$ in $U$,
    we see that 
    \begin{equation}\label{modbasis}
    (\til{f}X_1,\ldots, \til{f}X_l, X_{l+1},\ldots, X_n),
    \end{equation}
    for $\til{f}=s^*f$ the pullback of $f$ by the source map of $\Gg$, 
    forms a $C^\infty$--basis for the right hand side of~\eqref{lieofab}, 
    showing, incidentally, that it defines a locally free sheaf.  
	
	Along the exceptional divisor $E=\PP (N\Hh)$ of the blowup $\Bl{\Gg}{\Hh}$, the lifts 
	of the vector fields~\eqref{modbasis} have determinant given by 
	\[
	(p^*d\til{f}|_\Hh)^l\otimes X_1\wedge\cdots \wedge X_n,
	\]
	This defines a section of $\det T(\Bl{\Gg}{\Hh})|_E$ which vanishes to order $l$ along the 
	bundle of hyperplanes $(p^*d\til{f}|_{\Hh})^{-1}(0)\subset E$.  But this is precisely the intersection
	$E\cap \thickbar{s^{-1}(L)}$, which is removed in $\BGpd{\Gg}{\Hh} = \Bl{\Gg}{\Hh} \backslash (\thickbar{s^{-1}(L)} \cup \thickbar{t^{-1}(L)})$.  Hence the lifts of the vector fields~\eqref{modbasis} generate $\Lie(\BGpd{\Gg}{\Hh})$, as required.

\end{proof}

\begin{definition}\label{blpalg}
    Let $A\to M$ be a Lie algebroid, and $B\to L$ a Lie subalgebroid over the closed hypersurface $L\subset M$.  We define the \defw{elementary modification $\BAlg{A}{B}$ of $A$ along $B$} to be the Lie algebroid with sheaf of sections given by 
    \begin{equation}\label{algblwp}
    \BAlg{A}{B}(U) = \{X \in \Gamma(U,A) ~|~ X|_L \in \Gamma(U\cap L,B)\},
    \end{equation}
    for open sets $U\subset M$.  
\end{definition}
In view of this definition, Corollary~\ref{c: blow-up agd} may be rephrased to state that there is a canonical isomorphism
\[
\Lie(\BGpd{\Gg}{\Hh}) \cong \BAlg{\Lie(\Gg)}{\Lie(\Hh)}
\]
of Lie algebroids, whenever $\Hh\subset \Gg$ is a subgroupoid over a closed hypersurface.

\begin{example}
    Let $\Gg_0 \rightrightarrows \RR^2$ be the pair groupoid of
    $\RR^2$, and choose coordinates $(x,y)$. Let $L = \{x=0\}$.
    Then we obtain a subgroupoid $\Hh_0 = L\times L\subset \Gg_0$.  The Lie algebroid associated to the blow-up $\Gg_1 = \BGpd{\Gg_0}{\Hh_0}$ is
    \[
    \Lie(\Gg_1) = \left<x\del{x},~\del{y}\right>,
    \]
	which we recognize as the log tangent algebroid $T_L(\RR^2) = \BAlg{T(\RR^2)}{TL}$.

    Suppose we now blow up $\Gg_1$ along the codimension 2 subgroupoid
    $\Hh_1 = p_1^{-1}(\id_0(L))$, where $p_1:\Gg_1\to \Gg_0$ is the blow-down map.  The Lie algebroid corresponding to $\Gg_2 = \BGpd{\Gg_1}{\Hh_1}$ is then 
    \[
    \Lie(\Gg_2) = \left<x\del{x},~x\del{y}\right>,
    \]
	which is nothing but the Poisson algebroid $T^*_\pi\RR^2$, for $\pi = x\del{x}\wedge\del{y}$.

    On the other hand, we also have a codimension $3$ subgroupoid
    $\Hh'_1 = \id_1(L)$ of $\Gg_1$. The Lie algebroid corresponding to $\Gg'_2 = \BGpd{\Gg_1}{\Hh'_1}$ is then 
    \[
    \Lie(\Gg'_2) = \left<x^2\del{x},~x\del{y}\right>.
    \]
\end{example}

	\subsection{Blow-up of Poisson groupoids} \label{Construction}

	In this section, we construct a symplectic groupoid integrating a proper log
symplectic manifold $(M,\pi)$, by successively applying the blow-up operation to the pair groupoid $\Pair(M)$.  The fact that the blowup operations preserve both the Poisson and the groupoid structure is due, as we shall see, to the combined results of 
\S\ref{Section: Blow-upPoisson} and \S\ref{Section: Blow-upGpd}.
 
\subsubsection{The log pair groupoid}\label{logintegr}
Let $M$ be a manifold and $D\subset M$ a closed hypersurface with connected components $\{D_j\}_{j\in\sD}$. 
The source-connected subgroupoid of the pair groupoid of $D$ is then given by 
	\begin{equation}\label{diagsquare}
		\Pair^c(D) = \coprod_{j\in\sD} (D_j\times D_j).
	\end{equation}
\begin{definition} \label{Def: LogPairGpd} 
	Let $D\subset M$ be a closed hypersurface.  The \defw{log pair groupoid} $\Pair_D(M)$ of $(M,D)$ is defined by
    \[
    \Pair_D(M) = \BGpd{\Pair(M)}{\Pair^c(D)}^c;
    \]
    that is, it is the source-connected subgroupoid of the blow-up of $\Pair(M)$ along $\Pair^c(D)$. 
\end{definition}

By Corollary~\ref{c: blow-up agd}, the log pair groupoid $\Pair_D(M)$ integrates the Lie algebroid $\BAlg{TM}{TD}$ defined by~\ref{algblwp}, which is the log tangent bundle $T_D M$.

For manifolds $M$ with boundary $D$, the log pair groupoid $\Pair_D(M)$ is known as the {\it b}--stretched product~\cite{MR1348401} (See also~\cite{MR1394388}).  The groupoid $\Pair_D(M)$ is treated in greater depth in~\cite{MR1600121}.

\begin{theorem} \label{Prop: 1stBlowUpGpd}
       Let $(M, \pi)$ be a log symplectic manifold with degeneracy locus
        $D$, and let $p:\Pair_D(M)\to \Pair(M)$ be the blow-down groupoid morphism. 
        Then there is a unique log symplectic structure $\sigma$ on $\Pair_D(M)$ 
        such that $p_*(\sigma) = -\pi\times \pi$. This
        makes $(\Pair_D(M),\sigma)$ a Poisson groupoid over $(M, \pi)$, and the blow-down a morphism of Poisson groupoids.
\end{theorem}

\begin{proof}
    Recall that $\Pair_D(M)=\BGpd{\Pair(M)}{\Pair^c(D)}^c$, defined by 
    \[
    \blowup{\Pair(M)}{\Pair^c(D)}\backslash\thickbar{s^{-1}(D)}\cup \thickbar{t^{-1}(D)},
    \] 
    where $s, t$ are the source and target projections $M\times M\to M$.  
    But $s^{-1}(D) = M\times D$ and $t^{-1}(D) = D \times M$, so by
    Proposition~\ref{bbsympblowup}, we obtain the required log
    symplectic structure $\sigma$ on $\Pair_D(M)$ lifting $-\pi\times\pi$ on $\Pair(M)$.  
    Since the graph of the multiplication on $\Pair_D(M)$ is coisotropic on an open dense subset (i.e.\ the complement of the exceptional divisor), it follows that the graph must be everywhere coisotropic, proving $(\Pair_D(M),\sigma)$ is a Poisson groupoid.
\end{proof}

\subsubsection{The symplectic pair groupoid}

Let $(M, \pi)$ be a proper log symplectic manifold, so that each connected component $D_j$ of the degeneracy locus $D$ is a symplectic fibre bundle $f_j: D_j\to \gamma_j$ over a circle $\gamma_j\cong S^1$, as explained in~\S\ref{sympmaptorus}.  With $\gamma = \coprod_{j\in\sD} \gamma_j$, we obtain a projection map  
\begin{equation*}
f : D\to \gamma, 
\end{equation*}
which endows the pair $(M, D)$ with a structure akin to that of a manifold with fibred boundary~\cite{MR1734130}. The Lie algebroid $T^*_\pi M$ of such a Poisson structure is isomorphic to the Lie algebroid $T_{D,f}M$ whose sheaf of sections is given by 
\[
T_{D,f}(U) = \{X\in \Gamma(U,TU) ~|~ X|_D\in \Gamma(D,\ker Tf)\},
\]
and this algebroid may be expressed as a blow-up in the sense of Definition~\ref{blpalg}, as follows.  The log tangent algebroid $T_D M$ restricts to $D$ to define a Lie algebroid $T_DM|_{D}$, naturally isomorphic to the Atiyah algebroid $\At(ND)$ of infinitesimal symmetries of the normal bundle of $D$. The composition of the projection in the exact sequence 
\[
\xymatrix{0\ar[r] & \underline{\RR}\ar[r] & \At(ND)\ar[r] & TD\ar[r] & 0}
\]
with $Tf:TD\to T\gamma$ has a kernel $\At_f(ND)\subset\At(ND)$, which may be viewed as a relative Atiyah algebroid with respect to the fibration $f$.  Therefore we obtain the following representation of the Lie algebroid underlying $T^*_\pi M$ as a blow-up:
\begin{equation}\label{expresblow}
T_{D,f}M = \BAlg{T_D M}{\At_f(ND)}.
\end{equation}
By construction, $\At_f(ND)$ is a subalgebroid of $T_DM$, and we now construct the corresponding subgroupoid of $\Pair_D(M)$, which may be viewed as the gauge groupoid of $ND$ relative to the fibration $f$.  

Inside the pair groupoid $\Pair^c(D)$, we have the pair groupoid relative to $f$, given by 
\begin{equation*}
\Pair^c_f(D) = \coprod_{j\in\sD}(D_j \times_{\gamma_j}D_j).
\end{equation*}
Its preimage in the exceptional divisor for the blow-down $p:\Pair_D(M)\to\Pair(M)$ is  
\begin{equation}\label{secondblowupcenter}
\mathbf{G L}^+_f(ND) = p^{-1}\left(\coprod_{j\in\sD}( D_j {\times}_{\gamma_j}D_j)\right).
\end{equation}
This defines a subgroupoid $\mathbf{GL}^+_f(ND)\rightrightarrows D$, with Lie algebroid 
$\At_f(ND)$, which is a codimension 2 symplectic leaf in the log symplectic 
manifold $\Pair_D(M)$.  We now perform a blow-up to obtain a symplectic groupoid integrating $T^*_\pi M$.

\begin{theorem} \label{sympgpd} Let $(M, \pi)$ be a proper log symplectic
    manifold with degeneracy locus $D\subset M$, and let $\mathbf{GL}^+_f(ND)\rightrightarrows D$  be the subgroupoid of $\Pair_D(M)\rightrightarrows M$ defined in~\eqref{secondblowupcenter}.  Then the groupoid 
    \begin{equation}\label{defpari}
    \Pair_\pi(M) = \BGpd{\Pair_D(M)}{\mathbf{GL}^+_f(ND)}
    \end{equation}
    has a unique symplectic structure $\omega$ such that the blow-down $\Pair_\pi(M)\to\Pair_D(M)$ is Poisson.  This makes $(\Pair_\pi(M),\omega)$ a symplectic groupoid integrating $(M,\pi)$, and the blow-down a morphism of Poisson groupoids.
\end{theorem}

\begin{proof}
	Let $(\Pair_D(M), \sigma)\rightrightarrows (M,\pi)$ be the Poisson groupoid constructed in Theorem~\ref{Prop: 1stBlowUpGpd}, and let $s, t$ be its source and target maps.  The subgroupoid $\mathbf{GL}^+_f(ND)\rightrightarrows D$ is a symplectic leaf in the degeneracy locus of the log symplectic manifold $(\Pair_D(M),\sigma)$.  By Theorem~\ref{bsympblowup}, the blow-up 
	\[
\blowup{\Pair_D(M)}{\mathbf{GL}^+_f(ND)}\xrightarrow{p'} \Pair_D(M)
	\]
inherits a unique Poisson structure $\til{\sigma}$ such that $p'_*\til{\sigma} = \sigma$, and $\til{\sigma}$ is symplectic away from the proper transform of the degeneracy locus $s^{-1}(D) = t^{-1}(D)$ of $(\Pair_D(M),\sigma)$.  But this is precisely the locus removed in the definition of $\BGpd{\Pair_D(M)}{\mathbf{GL}^+_f(ND)}$. Hence the groupoid $\Pair_\pi(M)$ inherits a unique symplectic form compatible with the blow-down.  
It remains is to show that $\omega$ is multiplicative, but this is true by continuity, since it holds away from the exceptional divisor.
\end{proof}

\begin{remark}
	The blowup of a source-connected Lie groupoid along a source-connected subgroupoid 
	may fail to be source-connected, since the exceptional divisor in~\eqref{defngpdblowup}
	consists of a projective bundle with two families of hyperplanes removed, and the complement of a pair of hyperplanes in $\RR P^n$ is generically disconnected. In the 
	case of the symplectic pair groupoid~\eqref{defpari}, however, the deleted loci coincide,
	and so $\Pair_\pi(M)$, as defined, is source-connected. 
\end{remark}



	\subsection{Adjoint groupoids} \label{Section: adjoint}

	In this section, we prove that the source-connected groupoids $\Pair_D(M)$ and
$\Pair_\pi(M)$ constructed in \S\ref{Construction} are in fact the adjoint groupoids
integrating $T_D M$ and $T^*_\pi M$, respectively.  

\begin{theorem}\label{reladjoint}
Let $\Gg\rightrightarrows M$ be a Lie groupoid, and $\Hh\rightrightarrows L$ 
a closed Lie subgroupoid over the closed hypersurface $L\subset M$. Let $\Ff\rra M$ be a source-connected Lie groupoid and $\varphi:\Ff\to\Gg$ a morphism covering $\id_M$.  

If $\Lie(\varphi)$ factors through the elementary modification 
$\BAlg{\Lie(\Gg)}{\Lie(\Hh)}\to \Lie(\Gg)$, 
then there exists a unique groupoid morphism 
$\til{\varphi}:\Ff\to\BGpd{\Gg}{\Hh}$ 
completing the following commutative diagram.
\begin{equation}
\begin{aligned}
\xymatrix{ & \BGpd{\Gg}{\Hh} \ar[d]^-{p} \\
\Ff \ar[r]_-{\varphi} \ar@{.>}[ur]^-{\til{\varphi}}	& \Gg }
\end{aligned}
\end{equation}
\end{theorem}
\begin{proof}
	We show that $\varphi$ satisfies the criterion in 
	Proposition~\ref{univblowup2}, so lifts to a map $\til{\varphi}:\Ff\to \Bl{\Gg}{\Hh}$. 
	We then show that the image lies in $\BGpd{\Gg}{\Hh}$; 
	the remainder follows by continuity.
	
	Let $\fF, \gG, \hH$ be the Lie algebroids of $\Ff, \Gg, \Hh$, respectively.
	Observe that $L$ is a $\BAlg{\gG}{\hH}$--invariant submanifold, and since we have a 
	morphism $\fF\to\BAlg{\gG}{\hH}$, it is also a $\fF$--invariant submanifold.  
	Since $L$ is closed, this implies that it is a union of $\fF$--orbits, and since $\Ff$ is
	source-connected, $L$ is a union of $\Ff$--orbits. 
	So, if $s, t$ are the source and target of $\Ff$, we see that the full subgroupoid over $L$ is given by 
	\[
	\Ff|_{L} = s^{-1}(L)\cap t^{-1}(L) = s^{-1}(L),
	\]
	and since $s$ is a submersion, we conclude that 
	$\Ff|_{L}$ is a source-connected Lie subgroupoid of $\Ff$ of codimension 1.  
	In fact, $\Ff|_{L} = \varphi^{-1}(\Hh)$, which can be seen as follows.	
	Since $\Lie(\varphi)$ factors through $\Lie(p):\BAlg{\gG}{\hH}\to \gG$, we have 
	$\Lie(\varphi)(\fF|_L)\subset\hH$.  
	Therefore, the exponential map gives $\varphi(\Ff|_L)\subset \Hh$, 
	since $\Ff|_L$ is source-connected.  
	For the reverse inclusion, note that if $k\in\Ff\setminus\Ff|_L$, 
	then $s(k)\notin L$, so $\varphi(k)\notin \Hh$.  
	
	Since we have shown that $\varphi^{-1}(\Hh)$ is a hypersurface, 
	all that remains is to show $N\varphi|_{\Ff|_L}$ is injective.
	But this follows from the fact that $\varphi$ intertwines the 
	source maps of $\Ff$ and $\Gg$, which are submersions. 
\end{proof}

\begin{corollary} \label{MontAdjoint} 
	The log pair groupoid $\Pair_D(M)\rightrightarrows M$ associated to a closed hypersurface 
	$D \subset M$ is the adjoint integration of $T_D M$.
\end{corollary}
\begin{proof} 
    If $\Gg \rightrightarrows M$ is any source-connected integration
    of $T_D M$, then the source and target maps define a groupoid morphism 
    $\varphi = (t,s)$ from $\Gg$ to $\Pair(M)$, 
    whose Lie algebroid morphism is the anchor map of $T_D M$.  
    Since $T_DM = \BAlg{TM}{TD}$, Theorem~\ref{reladjoint} gives a lift 
    $\til\varphi: \Gg \rightarrow \Pair_D(M)$, completing the following commutative 
    diagram and establishing the claim.
		\begin{equation*}
    		\xymatrix{ 		& \Pair_D(M)\ar[d]^-{p}\\
						\Gg\ar[r]_-{\varphi}\ar@{.>}[ru]^-{{\til\varphi}} & \Pair(M)}
    	\end{equation*} 
\end{proof}
%

\begin{corollary} \label{SymGpdAdjoint}
    The symplectic pair groupoid $\Pair_\pi(M)$ of the proper log symplectic manifold $(M, \pi)$ is the adjoint integration of $T^*_\pi M$. 
\end{corollary}

\begin{proof}
	The degeneracy locus $D\subset M$ of the Poisson structure $\pi$ is a closed hypersurface, 
	and the anchor map $T_\pi M\to TM$ for the Poisson algebroid factors through the elementary 
	modification $T_D M\to TM$.  Therefore, if $\Gg\rightrightarrows M$ is any 
	source-connected integration of 
    $T^*_\pi M$, we use Theorem~\ref{reladjoint}, as in the proof 
    of Corollary~\ref{MontAdjoint}, to show that the canonical map
    $\varphi_0=(t,s):\Gg\to \Pair(M)$ lifts to a morphism $\varphi_1:\Gg\to \Pair_D(M)$, 
    and further to $\varphi_2:\Gg\to\Pair_\pi(M)$, establishing the result. 
\end{proof}
%
%

\begin{theorem} \label{thm: blcomp}
Let $\Gg\rra M$ be a Lie groupoid, and $\Kk\subset\Hh$ an inclusion of subgroupoids, each over the closed hypersurface $L\subset M$.  Then the proper transform $\thickbar{\Kk}$ is a Lie subgroupoid of $\BGpd{\Gg}{\Hh}$ over $L$, and we have a natural isomorphism of groupoids
\[
\BGpd{\Gg}{\Kk}^c \cong \BGpd{\BGpd{\Gg}{\Hh}}{\thickbar{\Kk}}^c.
\] 
\end{theorem}
\begin{proof}
We write the blow-down maps as $p: \BGpd{\Gg}{\Hh} \rightarrow \Gg$, $q: \BGpd{\BGpd{\Gg}{\Hh}}{\thickbar{\Kk}} \rightarrow \BGpd{\Gg}{\Hh}$ and $r: \BGpd{\Gg}{\Kk} \rightarrow \Gg$, all of which are Lie groupoid morphisms.

Since $\Kk \subset \Hh$, we have $\thickbar{\Kk} = p^{-1}(\Kk)$. Restrict the source $\til{s} = s \circ p: \BGpd{\Gg}{\Hh} \rightarrow M$ to $\thickbar{\Kk}$. Since $p: \thickbar{\Kk} \rightarrow \Kk$ and $s: \Kk \rightarrow L$ are submersions, it follows that $\til{s}: {\thickbar{\Kk}} \rightarrow L$ is submersion. Similarly, $\til{t}: {\thickbar{\Kk}} \rightarrow L$ is a submersion. Since $p$ is a groupoid morphism and $p(\thickbar{\Kk}) = \Kk$, it follows that $\til{\id}(L)$ and $\til{m}(\thickbar{\Kk}\times_L\thickbar{\Kk})$ both lie in $\thickbar{\Kk}$. This shows $\thickbar{\Kk}$ is a Lie subgroupoid of $\BGpd{\Gg}{\Hh}$.


Let $\gG$, $\hH$, $\kK$ be the Lie algebroids of $\Gg$, $\Hh$, $\Kk$, respectively. Since $\kK\subset\hH$ and $\hH$ is a Lie subalgebroid of $\gG$, it follows from Definition \ref{blpalg} that
\[
\BGpd{\gG}{\kK} = \BGpd{\BGpd{\gG}{\hH}}{\kK}.
\]

Since $\Lie(p \circ q): \BGpd{\BGpd{\gG}{\hH}}{\kK} \rightarrow \gG$ factors through $\BGpd{\gG}{\kK}$, by Theorem \ref{reladjoint}, we obtain a Lie groupoid morphism
\[
\varphi: \BGpd{\BGpd{\Gg}{\Hh}}{\thickbar{\Kk}}^c \rightarrow \BGpd{\Gg}{\Kk}.
\]
Likewise, $\Lie(r): \BGpd{\gG}{\kK} \rightarrow \gG$ factors through $\BGpd{\BGpd{\gG}{\hH}}{\kK}$, so we obtain the morphism
\[
\phi:\BGpd{\Gg}{\Kk}^c \rightarrow \BGpd{\BGpd{\Gg}{\Hh}}{\thickbar{\Kk}}.
\]
Since $\varphi \circ \phi: \BGpd{\Gg}{\Kk}^c \rightarrow \BGpd{\Gg}{\Kk}^c$ is a Lie groupoid morphism covering $\id$ on $\BGpd{\gG}{\kK}$, it follows that $\varphi \circ \phi$ is an automorphism of $\BGpd{\Gg}{\Kk}^c$, showing that $\BGpd{\Gg}{\Kk}^c \cong \BGpd{\BGpd{\Gg}{\Hh}}{\thickbar{\Kk}}^c$.
\end{proof}

\begin{corollary}
The symplectic pair groupoid $\Pair_\pi(M)$ may be alternatively constructed as $\BGpd{\Pair(M)}{\Pair_{f}^c(D)}^c$.
\end{corollary}
\begin{proof}
Recall that $\Pair_{f}^c(D) \subset \Pair^c(D) \subset \Pair(M)$ and the symplectic pair groupoid $\Pair_\pi(M)$ is constructed as
\begin{align*}
\Pair_\pi(M) & = \BGpd{\Pair_D(M)}{p^{-1}(\Pair_{f}^c(D))} \\
& = \BGpd{\BGpd{\Pair(M)}{\Pair^c(D)}^c}{p^{-1}(\Pair_{f}^c(D))} \\
& = \BGpd{\BGpd{\Pair(M)}{\Pair^c(D)}^c}{\thickbar{\Pair_{f}^c(D)}}.
\end{align*}
where $p: \Pair_D(M) \rightarrow \Pair(M)$ is the blow-down map. Since $\Pair_\pi(M)$ is source-connected, by Theorem \ref{thm: blcomp}, we obtain $\Pair_\pi(M) = \BGpd{\Pair(M)}{\Pair_{f}^c(D)}^c$.
\end{proof}


\section{Gluing and classification of groupoids} \label{Section: Gluing}
In this section, we present a general method for constructing Lie groupoids on a manifold by gluing together groupoids defined on the open sets of a covering. For this to be possible, the open cover must be adapted to the orbits of the lie algebroid in question; groupoids are inherently global objects and in general cannot be so easily decomposed. In~\S\ref{subsec: Gluing}, we consider a simple kind of cover, called an \emph{orbit cover}, which permits the result to hold. 

We then use this construction to explicitly describe the category of all groupoids integrating the log tangent bundle of a closed hypersurface $D\subset M$, as well as all Hausdorff symplectic groupoids integrating proper log symplectic manifolds in any dimension.  This involves solving a local classification problem near each component of the degeneracy locus, and then combining these local results in a specific manner using a graph constructed from the global geometry of the manifold and its embedded hypersurfaces. 

	\subsection{Orbit covers and gluing of groupoids} \label{subsec: Gluing}

	A typical way to construct manifolds is by the \emph{fibered coproduct} operation, also known as gluing.  If $M_1$, $M_2$ are manifolds equipped with open immersions $i_1:U\hookrightarrow M_1$, $i_2:U\hookrightarrow M_2$ from a manifold $U$, then the fibered coproduct is given by 
\[
M_1\coprod_{U} M_2 = \frac{M_1\coprod M_2}{i_1(x)\sim i_2(x)\ \forall  x\in U}.
\]
The caveat is that the resulting space is only Hausdorff when the graph of the equivalence relation above is closed in $M_1\times M_2$.  So, the fibered coproduct of manifolds is a possibly non-Hausdorff manifold.  

Suppose that $A$ is a Lie algebroid over a fibered coproduct of manifolds as above.  
We would like to construct 
a Lie groupoid integrating $A$
by gluing integrations over $M_1$ and $M_2$ using 
open immersions of an integration over $U$.  
But, groupoids are non-local, and 
such a simple gluing construction is not generally possible.  For example, the 
fundamental groupoid of the fibered coproduct should contain paths joining points in $M_1$ 
with points in $M_2$, and these may not be present in either of the fundamental 
groupoids of $M_1, M_2$.  In general, a groupoid coproduct operation is required,
whereby compositions absent from the naive gluing of spaces are formally 
adjoined.  However, we are able to avoid this complication, by using a
decomposition of the base which is adapted to the orbits of the Lie algebroid.
Heuristically, we are able to naively glue groupoids, but only along interfaces 
where they are actually local over the base.   

Essentially the same strategy was used by Nistor~\cite{MR1774632}
to obtain groupoids of interest in the theory of pseudodifferential operators; the setup 
we present below, while less general than his theory of $A$-invariant stratifications, 
is well-adapted for our purpose, which is to classify integrations of log tangent and log symplectic algebroids.    
  
\begin{definition}
    Let $A$ be a Lie algebroid over the manifold $M$.  An open cover $\{U_i\}_{i\in I}$ 
    of $M$ is called an \defw{orbit cover} if each orbit of the Lie algebroid 
    is completely contained in at least one of the open sets $U_i$.
\end{definition}

\begin{definition}
Given a groupoid $\Gg\rightrightarrows M$ and an open set $U\subset M$, we define the \defw{restriction} of the groupoid to $U$ to be the Lie subgroupoid
\[
\Gg|_U = s^{-1}(U)\cap t^{-1}(U).
\]
If $A$ is an integrable Lie algebroid, we define $\mathbf{R}_U$ to be the restriction functor from the category of source-connected integrations of $A$ to that of $A|_U$:
\begin{equation*}
	\begin{aligned}
	\mathbf{R}_U: 		\mathbf{Gpd}(A) &\rightarrow \mathbf{Gpd}(A|_U)\\
					 	(\Gg,\phi)&\mapsto ((\Gg|_U)^c, \phi|_U).				
	\end{aligned}
\end{equation*}
\end{definition}
\begin{remark}
If $S\subset M$ is a submanifold which is closed and $A$--invariant, meaning $a(A)|_S\subset TS$, then any Lie algebroid orbit $\Oo$ intersecting $S$ must be contained in $S$, and so $s^{-1}(S) = t^{-1}(S)$ for any source-connected groupoid $\Gg$ integrating $A$. For this reason, the restriction  
\[
\Gg|_S = s^{-1}(S) \cap t^{-1}(S) = s^{-1}(S)
\]
is a submanifold of $\Gg$, and so defines a source-connected Lie subgroupoid $\Gg|_S\rightrightarrows S$ of $\Gg$.
\end{remark}


\begin{theorem} \label{Thm: MfldvanKampen} 	
	Let $A$ be an integrable Lie algebroid over $M$, and 
	let $\{U_i\}_{i\in I}$ be a locally finite orbit cover of $M$.
    For each $i\in I$, let $\Gg_i \rightrightarrows U_i$ be a source-connected Lie
    groupoid integrating $A|_{U_i}$, and for each $i, j\in I$, let 
    \[
    \phi_{ij}: \R_{U_i\cap U_j}(\Gg_i) \stackrel{\cong}{\longrightarrow} \R_{U_i\cap U_j}(\Gg_j)
    \]
    be an isomorphism of integrations of $A|_{U_i}$, such that $\phi_{ii} =\id$, $\phi_{ij}=\phi_{ji}^{-1}$, and $\phi_{jk}\phi_{ij}=\phi_{ik}$ for all $i, j, k \in I$, on the appropriate intersections.  This defines an equivalence relation, whereby $x\sim \phi_{ij}(x)$ for all $x\in \R_{U_i\cap U_j}(\Gg_i)$ and for all $i, j\in I$.  Then, we have the following:
    
\begin{enumerate}
\item  
	The fibered coproduct of manifolds
    \begin{equation}\label{gpdglued}
    \Gg = \left.{\coprod_{i\in I} \Gg_i}\right/\sim
    \end{equation}
    is a source-connected Lie groupoid integrating $A$, 
    such that $\R_{U_i}(\Gg)=\Gg_i$.   
\item 
	The inclusion morphisms $\iota_k:\Gg_k\hookrightarrow \Gg,\ k\in I$ make $\Gg$ 
    a groupoid coproduct, meaning that for any groupoid which receives compatible 
    morphisms from $\{\Gg_i\}_{i\in I}$, these morphisms must factor 
    through a uniquely defined morphism from $\Gg$.  
\item 
	Every source-connected groupoid integrating $A$ is the fibered (manifold) coproduct of its
    restrictions to the orbit cover $\{U_i\}_{i\in I}$.  
\end{enumerate}
\end{theorem}

\begin{proof}
	The fibered coproduct $\Gg$ in~\eqref{gpdglued} immediately inherits 
	submersions $s, t :\Gg\to M$, and the embedding $\id:M\to\Gg$ from 
	the corresponding maps on the component groupoids $\Gg_i$, by the 
	universal property of coproducts.  For example, the source maps 
	$s_i:\Gg_i\to U_i$ and $s_j:\Gg_j\to U_j$ satisfy $s_i = s_j \phi_{ij}$ 
	as maps $\R_{U_i\cap U_j}(\Gg_i)\to U_i\cap U_j$, so by the universal 
	property we obtain a coproduct map 
	\[
		s_i\cup s_j :  \Gg_i\coprod_{\phi_{ij}} \Gg_j \to U_i\cup U_j.
	\] 
	In the same way, the inverse maps on each $\Gg_i$ glue to a map $i:\Gg\to\Gg$.  
	Less obvious is the fact that the multiplication maps $m_i$ of 
	each groupoid $\Gg_i$ glue to a map 
	\[
	m:\Gg {_{s}\times}_t\Gg\to \Gg.
	\] 
	To see this, we must use the orbit cover property, as follows.
	
	Let $(h,g)\in \Gg_i\times\Gg_j$ be a representative for an arbitrary point
	in $\Gg {_{s}\times}_t\Gg$, so that $s_i(h)=t_j(g)=x$.   	
	Now, $t_i(s_i^{-1}(x))$ is the $\Gg_i$--orbit of $x$, which coincides with 
	the $A|_{U_i}$--orbit of $x$, since $\Gg_i$ is source-connected. Therefore,
	$t_i(s_i^{-1}(x))$ sits in the full $A$--orbit of $x$ which, by the orbit cover property,
	must be contained in some $U_k$, $k\in I$.  
	Therefore, we have 
	\[
		t_i(s_i^{-1}(x))\subset U_i\cap U_k,
	\]
	proving that $s_i^{-1}(x)$ coincides with the source fiber $s_{ik}^{-1}(x)$ of the 
	groupoid $\Gg_i|^c_{U_i\cap U_k}$.  But this is identified in the coproduct with
	the source fibre of $(\Gg_k|_{U_i\cap U_k})^c$ using $\phi_{ik}$.
	Therefore, in $\Gg$, the element $h$ has a representative in $\Gg_k$. 	
%
%
%
	By the same argument, since $t_j(g)=s_i(h)$, $g$ also has a representative in $\Gg_k$, 
	hence we may use the given multiplication $m_k$ on $\Gg_k$ to define $m$
	in a neighbourhood of $(h,g)$. The compatibility of the component multiplications 
	ensures that this defines $m$ unambiguously.  
	
	The argument above also shows that $\Gg$ is source-connected, 
	since the source fibre $s^{-1}(x)$ coincides with 
	the source fiber of a subgroupoid $\Gg_k$ 
	such that $U_k$ contains the $A$--orbit of $x$, and $\Gg_k$ is source-connected.

	Part \emph{ii)} follows from the universal property of the manifold coproduct, 
	together with the fact that $\varphi:\Gg\to \Gg'$ is a Lie groupoid morphism
	if and only if $\varphi\circ\iota_k$ is a Lie groupoid morphism for all $k\in K$, 
	which follows from the local definition of the groupoid structure maps of $\Gg$.
	
	For Part \emph{iii)}, we let $\Gg'\in\mathbf{Gpd}(A)$ be any 
	source-connected groupoid integrating $A$, and 
	let $\Gg$ be the fibered coproduct of $\R_{U_i}(\Gg')$ given by~\eqref{gpdglued}.  
	By Part \emph{ii)}, we obtain a morphism $\varphi:\Gg\to\Gg'$ such that 
	$\R_{U_i}(\varphi)$ is an isomorphism. But $\varphi$ must then be an isomorphism, 
	since for any $x\in M$ as above, the restriction of $\varphi$ to $s^{-1}(x)=s_k^{-1}(x)$ 
	is an isomorphism.
\end{proof} 

\begin{example}
Let $p\in S^1$ be a point on the circle, and $T_p S^1$ the Lie algebroid of 
vector fields vanishing at $p$. We may construct a Lie groupoid 
integrating this algebroid as follows.  
Express the circle as a fibered coproduct of $U=\RR$, $V=\RR$, with gluing map
$\phi:U\setminus \{0\}\to V\setminus\{0\}$ given by $x\mapsto x^{-1}$:
\[
S^1 = U\coprod_\phi V.
\]
Let $p = 0\in U$.  Then $\{U,V\}$ is an orbit cover, since the orbits are $p\in U$ and $S^{1}\setminus\{p\} = V$.  A source-connected integration over $U$ is given by the action 
groupoid $\Gg_U = \RR\ltimes U$ of $\RR$ on $U$ by rescaling, with source and target maps 
\[
\begin{aligned}
s_U&:(t,x)\mapsto x\\
t_U&:(t,x)\mapsto e^t x.
\end{aligned}
\]  
Over $V$, the algebroid $T_pS^1$ is simply the tangent bundle, so a source-connected integration is $\Pair(V)=V\times V$.  We now glue $\Gg_U$ to $\Gg_V$ via the map 
\[
(t,x)\mapsto (e^{-t} x^{-1}, x^{-1}),
\]
an isomorphism of subgroupoids from $\RR\ltimes (U\setminus\{0\})$ to $(\Pair(V\setminus 0))^c = \Pair(V_+)\times\Pair(V_-)$, where $V_\pm=\{y\in\RR ~|~ \pm y>0\}$.  The resulting groupoid is a source-connected integration of $T_p S^1$, diffeomorphic as a smooth surface to the nontrivial line bundle over $S^1$.
\end{example}

\begin{remark}\label{Remark: MfldvanKampen}
The restriction functors associated to an open cover $\{U_i\}_{i\in I}$ define a functor from $\mathbf{Gpd}(A)$ to the fiber product of the categories $\mathbf{Gpd}(A|_{U_i})$ over $\mathbf{Gpd}(A|_{U_i\cap U_j})$.  A restatement of Theorem~\ref{Thm: MfldvanKampen} is that when the open cover is an orbit cover, the fibered coproduct of \emph{manifolds} defines an inverse functor to this restriction.  Concretely, if $\{U,V\}$ is an orbit cover, then we obtain the following fibre product diagram of categories.
    \begin{equation*} 
        \begin{aligned}\xymatrix{
                \Gpd(A) \ar[r]^-{\mathbf{R}_{U}} \ar[d]_-{\mathbf{R}_{V}}		& \Gpd(A|_{U}) \ar[d]^-{\mathbf{R}_{U \cap V}}  \\
                \Gpd(A|_{V}) \ar[r]_-{\mathbf{R}_{U \cap V}}							& \Gpd(A |_{U \cap V})	 }\end{aligned}
    \end{equation*}
Since the cocycle condition is trivially satisfied in this case, we may classify integrations on $U\cup V$ by classifying integrations on $U$ and $V$ which are isomorphic along $U\cap V$ (isomorphisms between integrations are unique when they exist). 
\end{remark}
Combining Theorem~\ref{Thm: MfldvanKampen} with Theorem~\ref{Thm: Gpd-NormalSubgpd}, we obtain a version of the gluing theorem stated in terms of normal subgroupoids of local source-simply-connected integrations.
\begin{corollary}\label{posetprod}
Let $A$ be a Lie algebroid over $M$, let $\{U_i\}_{i\in I}$ be an orbit cover for $M$, 
and let $\til{\Gg}_i$, $\til{\Gg}_{ij}$ be source-simply-connected integrations 
of $A$ over $U_i$ and $U_i\cap U_j$, respectively, for all $i, j\in I$.  
The canonical morphisms $\til{\Gg}_{ij}\to \mathbf{R}_{U_i\cap U_j}(\til\Gg_i)$ 
induce morphisms $\mathbf{P}_{i}:\Lnorm(\til{\Gg_i})\to \Lnorm(\til{\Gg}_{ij})$ between 
posets of discrete, totally disconnected, normal Lie subgroupoids.

Then the category of integrations $\Gpd(A)$ is equivalent to the fibre product of posets 
$\Lnorm(\til{\Gg_i})$ over the maps $\mathbf{P}_i$, i.e., the following limit:
\[
\Gpd(A) \simeq \lim \left(
	\prod_{i\in I} \Lnorm(\til{\Gg}_i) \xrightarrow{\mathbf{P}}
	\prod_{i,j\in I} \Lnorm(\til{\Gg}_{ij})
\right).
\]
In other words, any integration is uniquely specified by a choice of discrete, totally disconnected, normal Lie subgroupoids $\Nn_i\subset \til{\Gg}_i$, for all $i\in I$, such that $\mathbf{P}_i(\Nn_i) = \mathbf{P}_j(\Nn_j)$ in $\til{\Gg}_{ij}$.
\end{corollary}

\begin{prop} \label{HausGlue}
		Let $A$ be an integrable Lie algebroid over $M$, and let $\{U_i\}_{i\in I}$ be a locally finite orbit cover of $M$. If for each $i\in I$, the groupoid $\Gg_i
	\rightrightarrows U_i$ in Theorem \ref{Thm: MfldvanKampen} is Hausdorff, then the coproduct groupoid $\Gg \rra M$ is Hausdorff.
\end{prop}
\begin{proof}
For $g, h \in \Gg$ such that $g \neq h$, if we have $s(g) = s(h)$ and $t(g) = t(h)$, then since $\Gg$ is source-connected, there exists an orbit $\Oo$ of $A$ containing both $s(g)$ and $t(g)$. By the orbit cover property, $\Oo$ is contained in some $U_i$. It follows that $g, h \in \R_{U_i}(\Gg) = \Gg_i$. Since $\Gg_i$ is an embedded Hausdorff submanifold of $\Gg$, we have that $g$ and $h$ are separable.

If $x = s(g)$ and $y = s(h)$ are distinct, then we can find a neighbourhood $W_x$ of $x$ and a neighbourhood $W_y$ of $y$ such that $W_x \cap W_y =\varnothing$, because the base $M$ is Hausdorff. Since $s: \Gg \rightarrow M$ is a submersion, it follows that $s^{-1}(W_x) \subset \Gg$ and $s^{-1}(W_y) \subset \Gg$ are open sets such that $s^{-1}(W_x) \cap s^{-1}(W_y) = \varnothing$. Likewise, if $t(g)$ and $t(h)$ are distinct, then $g$ and $h$ are separable.
\end{proof}

	\subsection{Classification of integrations} \label{Subsection: Tangent}

	In this section we make use of Theorem~\ref{Thm: MfldvanKampen} to classify  integrations of the Lie algebroid $T_D M$ associated to a closed hypersurface $D\subset M$, as well as the lie algebroid $T^*_\pi M$ of a log symplectic structure. 

\subsubsection{Choosing an orbit cover}\label{graphorbit}

For both log tangent and log symplectic cases, we choose an orbit cover for the manifold $M$ as follows: $V$ is the complement of the closed hypersurface $D$, and $U$ is a tubular neighbourhood of $D$, chosen so that the tubular neighbourhoods of different connected components of $D$ do not intersect.

We index connected components as follows: let $\sD = \pi_0(D)$ and $\sV = \pi_0(V)$, so that 
\[
V = \coprod_{i\in \sV} V_i,\qquad
U = \coprod_{j\in \sD} U_j.
\] 
It is convenient to partition $\sD$ into two subsets, $\sD = \sE\coprod\sH$, 
where $\sE, \sH$ are the sets of connected components of $D$ with orientable and non-orientable 
normal bundles, respectively. It is also convenient, following~\cite{MR1959058}, to represent this information as a graph, as follows.

\begin{definition}\label{graphmd}
With notation as above, the \defw{graph of $(M, D)$} is the following graph with 
half-edges (a half-edge is an edge with only one end attached to a vertex).
\begin{itemize}
\item[--] The vertices, $\sV = \pi_0(V)$, index the components of the complement of $D$.  
\item[--] The edges, $\sE$, index the components of $D$ with orientable normal bundle; 
an edge $j\in \sE$ joins the pair of vertices representing the open components on either side of $D_j$ (note that these may coincide, in which case the edge becomes a loop). 
\item[--] The half-edges, $\sH$, index the components of $D$ with non-orientable normal bundle; 
a half-edge $j\in\sH$ is attached to a vertex $i\in\sV$ if $D_j\subset \overline{V_i}$.
\end{itemize}
\end{definition}

\begin{example}\label{ellcurve2}
Example~\ref{Example: EllipticCurve} concerns a hypersurface $D\subset \RR P^2$ with two connected components. We choose an orbit cover consisting of the complement $V=\RR P^2\setminus D$, with two connected components, and a tubular neighbourhood $U$ of $D$, with two connected components.  
\begin{center}
\centering\begin{tikzpicture}
\node [circle,fill=black,inner sep=1.2pt,pin={[pin distance=0.5cm,pin edge={black}]left:}] (V1) {};
\node [circle,fill=black,inner sep=1.2pt] (V0) [right =of V1] {};
\path (V0) edge [in=0,out=180] node [label=above:] {} (V1);
\end{tikzpicture}
\end{center}
The corresponding graph, shown above, has two vertices, one edge, and one half-edge, as $D$ has two connected components, one of which has nontrivial normal bundle.
\end{example}

The orbit cover described above has the property that the Lie algebroid $A$ (which is either $T_DM$ or $T^*_\pi M$)  restricts to the tangent algebroid on $V$ and $U\cap V$, i.e. $A|_V=TV$ and $A|_{U\cap V}=T(U\cap V)$.    So, the category of integrations of $A$ can be described as the following fibre product.
    \begin{equation}\label{commclass}
        \begin{aligned}
        \xymatrix{\Gpd(A) \ar[r] \ar[d]	& \Gpd(A|_U) \ar[d]^-{\mathbf{P}_U}  \\
                \Gpd(TV) \ar[r]_-{\mathbf{P}_{V}}	& \Gpd(T({U \cap V}))	 }
        \end{aligned}	
    \end{equation}
The fundamental groupoids of $V$ and $U\cap V$ provide two of the source-simply-connected integrations required to apply Corollary~\ref{posetprod}.
As described in Example~\ref{ex: tanggpd}, 	we may further simplify the bottom row 
of~\eqref{commclass} by restricting $\Pi_1(V)$ and $\Pi_1(U\cap V)$ to a set of 
basepoints for the underlying spaces, described in the next section,~\S\ref{bspt}.  This will render $\mathbf{P}_V$ into a morphism between posets of normal subgroups of the fundamental groups $\pi_1(V), \pi_1(U\cap V)$. The choice of basepoints will also be convenient for the description of $\mathbf{P}_U$ in~\S\ref{ltclass} and~\S\ref{lsclass}.

\subsubsection{Choosing basepoints}\label{bspt}

\begin{figure}[H]
\centering
\begin{tikzpicture}[>=angle 60]
\draw (0,1) .. controls (-.1,1) and (-.2,.5) .. (-.2,0)  .. controls (-.2,-0.5) and (-.1,-1) .. (0,-1);
\node[label=90:{\small $D_j$}] at (-.45,0) {};

\draw[shift={(-1,0)}, dotted] (0,1) .. controls (-.1,1) and (-.2,.5) .. (-.2,0)  .. controls (-.2,-0.5) and (-.1,-1) .. (0,-1);

\draw[shift={(+1,0)},dotted] (0,1) .. controls (-.1,1) and (-.2,.5) .. (-.2,0)  .. controls (-.2,-0.5) and (-.1,-1) .. (0,-1);

\draw (-5,1) -- (5,1);

\draw (-5,-1) -- (5,-1);

\coordinate [label=110:$x_j$] (x) at (0.4,-0.8);
\node [circle, fill=black,inner sep=1pt] at (-.14,-0.7) {};
	\coordinate [label=110:$x_{ji'}$] (x1) at (1.2,-0.8);
	\node [circle, fill=black,inner sep=1pt] at (.4,-0.7) {};
	\coordinate [label=110:$x_{ji}$] (x2) at (-.48,-0.8);
	\node [circle, fill=black,inner sep=1pt] at (-.68,-0.7) {};
\coordinate [label=right:{$y_{i'}$}] (y+) at (3,-0.7);
\node [circle, fill=black,inner sep=1pt] at (y+) {};
\coordinate [label=left:{$y_{i}$}] (y-) at (-3,-0.7);
\node [circle, fill=black,inner sep=1pt] at (y-) {};

\coordinate [label={\small $V_{i'}$}] (v+) at (3,0.2);
\coordinate [label={\small $V_{i}$}] (v-) at (-3,0.2);

\genhole{shift={(2,0)}} 
\draw[shift={(3,0)}] (0,0) node {$\cdots$};
\genhole{shift={(4,0)}}

\genhole{shift={(-2,0)}} 
\draw[shift={(-3,0)}] (0,0) node {$\cdots$};
\genhole{shift={(-4,0)}}

\draw [decorate,decoration={brace,amplitude=1ex},yshift=-.5ex]
(1,-1) -- (-1,-1) node [black,midway,yshift=-2.5ex] 
{\small $U_j$};
\end{tikzpicture}


\begin{tikzpicture}[>=angle 60]
\draw[middlearrow={>}] (0,1) .. controls (-.1,1) and (-.2,.5) .. (-.2,0)  .. controls (-.2,-0.5) and (-.1,-1) .. (0,-1);
\draw[middlearrow={<}] (0,1) .. controls (.1,1) and (.2,.5) .. (.2,0) node[right] {$D_j$}.. controls (.2,-0.5) and (.1,-1) .. (0,-1);

\draw (-5,1) -- (0,1);
\draw (-5,-1) -- (0,-1);

\coordinate [label=110:$x_j$] (x) at (-.07,-0.8);
\node [circle, fill=black,inner sep=1pt] at (-.14,-0.7) {};
\coordinate [label=110:$x_{ji}$] (x2) at (-.68,-0.8);
	\node [circle, fill=black,inner sep=1pt] at (-.78,-0.7) {};
\coordinate [label=left:{$y_i$}] (y-) at (-3,-0.7);
\node [circle, fill=black,inner sep=1pt] at (y-) {};

\coordinate [label={\small $V_i$}] (v+) at (-3,0.2);
\genhole{shift={(-2,0)}} 
\draw[shift={(-3,0)}] (0,0) node {$\cdots$};
\genhole{shift={(-4,0)}}
\draw[shift={(-1,0)}, dotted] (0,1) .. controls (-.1,1) and (-.2,.5) .. (-.2,0)  .. controls (-.2,-0.5) and (-.1,-1) .. (0,-1);
\draw [decorate,decoration={brace,amplitude=1ex},yshift=-.5ex]
(0,-1) -- (-1,-1) node [black,midway,yshift=-2.5ex] {\small $U_j$};
\end{tikzpicture}
\caption{Choice of basepoints in the cases $j\in\sE$ (above) and $j\in\sH$ (below).}
\label{basept}
\end{figure}

We make the following choice of basepoints, as illustrated in Figure~\ref{basept}:
\begin{itemize}
\item[--] For each $i\in\sV$, choose $y_i\in V_i$.
\item[--] For each $j\in \sD$, choose $x_j\in D_j$.
\item[--] For each $j\in \sE$, choose basepoints $x_{ji}, x_{ji'}$ in $V_i\cap U$ and $V_{i'}\cap U$ respectively, where $V_i, V_{i'}$  are the open components on either side of $D_j$.  Choose these in such a way that they are sent to $x_j$ by a neighbourhood retraction $r_j:U_j\to D_j$.
\item[--] For each $j\in \sH$, choose a basepoint $x_{ji}\in V_i\cap U$, where $V_i$ is the open component surrounding $D_j$.  Choose it so that it is sent to $x_j$ by a neighbourhood retraction $r_j:U_j\to D_j$. 
\end{itemize}

Once basepoints are chosen as above, we obtain a simplification of the bottom row of
the fiber product diagram~\eqref{commclass}.  Namely, we obtain equivalences
\begin{equation*}
	\begin{aligned}
		\Gpd(TV) &\simeq \prod_{i\in \sV}\Lnorm(\pi_1(V_i, y_i)),\\
		\Gpd(T(U\cap V)) &\simeq \prod_{i\in \sV,j\in \sD}\Lnorm(\pi_1(V_i\cap U_j, x_{ji})).
	\end{aligned}
\end{equation*}
With respect to this decomposition, the restriction functor $\mathbf{P}_V$ has the following simple description, by an argument as in Example~\ref{ex: tanggpd}.
\begin{prop}
The restriction functor $\mathbf{P}_V$ taking integrations of $TV$ to integrations of $T(U\cap V)$ may be described as a poset map from  normal subgroups of $\pi_1(V_i, y_i)$ to 
normal subgroups of $\pi_1(U_j\cap V_i, x_{ji})$: it is the pullback by the group homomorphism 
\[
\delta_*:\pi_1(U_j\cap V_i, x_{ji})\to \pi_1(V_i, y_i),\qquad \gamma \mapsto \delta \gamma \delta^{-1}
\] 
induced by the choice of a path $\delta$ from $y_i$ to $x_{ji}$ in $V_i$.  This map on normal subgroups $N\mapsto \delta_*^{-1}(N)$ is independent of the choice of $\delta$.
\end{prop}

To obtain a complete description of $\Gpd(A)$, all that remains is to describe $\Gpd(A_U)$ and $\mathbf{P}_U$ in~\eqref{commclass}.  
Since $U$ is the disjoint union of tubular neighbourhoods $U_j$ of components $D_j,\ j\in\sD$, the problem reduces to a local investigation: we need only describe the source-simply-connected groupoid integrating $A_{U_j}$  and its poset of discrete, totally disconnected normal Lie subgroupoids. 

\subsubsection{Log tangent integrations}\label{ltclass}
Fix a tubular neighbourhood $U_j$ of a single connected component $D_j$ of the hypersurface $D$, and choose basepoints $x_j$, $x_{ji}$ and, if $j\in\sE$, $x_{ji'}\in U_j\cap V_{i'}$, as described in~\S\ref{bspt}.

In Appendix~\ref{Appen: LogTang}, we construct the source-simply-connected groupoid $\til{\Gg}_{U_j}$ 
integrating $T_{D_j} U_j$, compute its poset of 
discrete, totally disconnected normal Lie subgroupoids 
(as well as the subposet of closed subgroupoids), and describe the restriction functor
$\mathbf{P}_{U_j}$.  The results of Propositions~\ref{Prop: MonthubertLocal} and~\ref{Prop: MonthubertLocalNOr} are summarized as follows.

\begin{theorem}[Local classification]\label{locA}
Let $D_j$, $U_j$, and $V_i$ be as in~\S\ref{graphorbit} and choose basepoints as in~\S\ref{bspt}.
\noindent
\begin{enumerate}
\item
If $D_j$ has orientable normal bundle, then the integrations of $T_{D_j} U_j$ are classified by triples
\begin{equation}\label{clalt}
(K_{ji}, K_j, K_{ji'}) 
\end{equation}
of normal subgroups $K_{ji}\subset \pi_1(U_j\cap V_i, x_{ji})$, $K_j\subset \pi_1(D_j, x_j)$, and $K_{ji'}\subset \pi_1(U_j\cap V_{i'}, x_{ji'})$, which are compatible with the projection $r:U_j\setminus D_j\to D_j$ of the punctured tubular neighbourhood, in the sense 
\begin{equation}\label{condi1}
K_j \subset  r_*K_{ji}\ and\ K_j\subset r_*K_{ji'}.
\end{equation}

\item If $D_j$ has non-orientable normal bundle, then the integrations of $T_{D_j} U_j$ are classified by pairs
\((K_{ji}, K_j)\) of normal subgroups as above, such that 
\begin{equation}\label{condi2}
	K_j\subset r_*K_{ji}.
\end{equation}

\item Morphisms between integrations correspond to componentwise inclusion for the associated triple (or pair) of normal subgroups.

\item Restricting the integration of $T_{D_j} U_j$ given by~\eqref{clalt} to $U_j\cap V_{i}$, we obtain the integration of $T(U_j\cap V_i)$ defined by
\begin{equation}\label{quotclas}
\left.\Pi_1(U_j\cap V_i) \right/\Nn_{ji},
\end{equation}
where $\Nn_{ji}$ is the unique totally disconnected normal Lie subgroupoid with isotropy $K_{ji}$ at $x_{ji}$.  
Similarly, in the orientable case, the restriction to $U_j\cap V_{i'}$ yields $\Pi_1(U_j\cap V_{i'})/\Nn_{ji'}$, where $\Nn_{ji'}$ is the subgroupoid with isotropy  $K_{ji'}$ at $x_{ji'}$.

\item The fundamental group of the source fibre over $x_j\in D_j$ is isomorphic to $K_j$; in particular, the source-simply-connected integration is obtained when all subgroups in the triple (or pair) are trivial. 

\item Hausdorff integrations are those for which the inclusions~\eqref{condi1}, \eqref{condi2} are equalities.

\end{enumerate}
\end{theorem}
\begin{remark}\label{identik}
If the normal bundle of $D_j$ is orientable, $r_*$ is an isomorphism, so we may view the groups~\eqref{clalt} as subgroups of the same group $\pi_1(D_j, x_j)$. Consequently, condition~\eqref{condi1} is simply that the normal subgroup $K_j$ must lie in the intersection $K_{ji}\cap K_{ji'}$. For Hausdorff integrations, all three groups must coincide.  

In the non-orientable case, $r_*$ is an injection of $\pi_1(U_j\cap V_i, x_{ji})$ onto the kernel of the first Stiefel-Whitney class $w_1:\pi_1(D_j, x_j)\to \ZZ/2\ZZ$ of the normal bundle of $D_j$.  So, we may view $K_{ji}$ as a normal subgroup of $\ker w_1$, and condition~\eqref{condi2} then states that  
$K_j\subset K_{ji}$.
In the Hausdorff case, this is an equality (in particular, this implies $K_{ji}$ is normal in 
$\pi_1(D_j,x_j)$).
\end{remark}

With Theorem~\ref{locA}, we are able to fill in the diagram~\eqref{commclass} and give a global description of the category of integrations $\Gpd(T_D M)$. We will phrase the fibre product in terms of the graph introduced in~\S\ref{graphorbit}, using the basepoint choices from~\S\ref{bspt}.  

\begin{definition}\label{gragro}
The graph of groups associated to $(M,D)$ is defined as follows. Let $\Gamma$ be the graph associated to $(M,D)$ in Definition~\ref{graphmd}.  Let $\delta_{ji}$ be paths joining $y_i$ to $x_{ji}$ for all $i\in\sV, j\in\sD$. We label $\Gamma$ with groups and homomorphisms in the following way, using the identifications in Remark~\ref{identik}.  

\begin{enumerate}
\item[--] To each vertex $i\in\sV$, we associate the group $\pi_1(V_i, y_i)$.
\item[--] To each edge $j\in\sE$ joining $i$ to $i'$, we associate the group 
$\pi_1(D_j, x_j)$, together with the induced homomorphisms 
$(\delta_{ji})_*, (\delta_{ji'})_*$ from $\pi_1(D_j, x_j)$ to the corresponding
vertex groups $\pi_1(V_i, y_i)$ and $\pi_1(V_{i'}, y_{i'})$.
\begin{center}
\centering
\begin{tikzpicture}
\node (mid) at (0,1) {$\pi_1(D_j)$};
\node (left) at (-2, .4) {$\pi_1(V_i)$};
\node (right) at (2, .4) {$\pi_1(V_{i'})$};
\node [circle,fill=black,inner sep=1pt] at (2,0) {};
\node [circle,fill=black,inner sep=1pt] at (-2,0) {};
\path (mid) edge [->] (right);
\path (mid) edge [->] (left);
\path (2,0) edge [in=0,out=180] node [] {} (-2,0);
\end{tikzpicture}
\end{center}
\item[--] To each half-edge $j\in\sH$ attached to $i$, we associate the inclusion of groups $\ker w^j_1\hookrightarrow \pi_1(D_j,x_j)$ determined by the Stiefel-Whitney class $w^j_1$ of $ND_j$, together with the induced homomorphism $(\delta_{ji})_*:\ker w^j_1\to \pi_1(V_i,y_i)$. 
\begin{center}
\centering
\begin{tikzpicture}
\node (up) at (-1,1.3) {$\ker w_1^j$};
\node (mid) at (0,.4) {$\pi_1(D_j)$};
\node (left) at (-2, .4) {$\pi_1(V_i)$};
\node [circle,fill=black,inner sep=1pt] at (-2,0) {};
\path (up) edge [->] (mid);
\path (up) edge [->] (left);
\path (0,0) edge [in=0,out=180] node [] {} (-2,0);
\end{tikzpicture}
\end{center}
\end{enumerate}
\end{definition}

\begin{theorem}[Global classification]\label{classifylogtan}
Given the graph of groups associated to $(M,D)$ in Definition~\ref{gragro}, the category of integrations $\Gpd(T_D M)$ is equivalent to the poset whose elements consist of:
\begin{enumerate}
\item A normal subgroup $K_i$ of each vertex group $\pi_1(V_i, y_i), i\in\sV$,
\item A normal subgroup $K_j$ for each edge group  $\pi_1(D_j, x_j), j\in\sE$, such that 
\begin{equation}\label{edgecon}
K_j\subset \delta_{ji}^{-1}(K_i)\text{ and } K_j\subset \delta_{ji'}^{-1}(K_{i'}),
\end{equation}
where $j$ joins $i$ to $i'$, 
\item A normal subgroup $K_j$ of each half-edge group $\pi_1(D_j, x_j), j\in\sH$, such that 
\begin{equation}\label{halfedgecon}
K_j\subset \delta_{ij}^{-1}(K_i)\ \ \text{ in }\ \ \ker w_1^j,
\end{equation}
where $j$ is attached to $i$.
\end{enumerate}
The partial order is componentwise inclusion for the corresponding normal subgroups, and the fundamental group of the source fiber over any of the basepoints is given by Theorem~\ref{locA}.  In particular, the source-simply-connected integration is obtained when all subgroups over vertices, edges, and half-edges are trivial.
Finally, the Hausdorff integrations are for which the inclusions in~\ref{edgecon} and~\ref{halfedgecon} are all
equalities.
\end{theorem}

\begin{example}
The log tangent integrations for Example~\ref{ellcurve2} are classified using the following graph 
of groups:
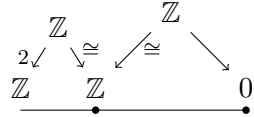
\begin{center}
\begin{tikzpicture}
	\node (mid) at (-1,.3) {$\ZZ$};
	\node (right) at (1, .3) {$0$};
	\node (up) at (-1.5,1.1) {$\ZZ$};
	\node (left) at (-2, .3) {$\ZZ$};
	\node (edge) at (0,1.3) {$\ZZ$};
	\node [circle,fill=black,inner sep=1pt] at (-1,0) {};
	\node [circle,fill=black,inner sep=1pt] at (1,0) {};
	\path (up) edge [->] node[very near start, right] {\scriptsize $\cong$} (mid);
	\path (up) edge [->] node[left] {\scriptsize $2$} (left);
	\path (edge) edge [->] node[right] {\scriptsize $\cong$} (mid);
	\path (edge) edge [->] (right);
	\path (-2,0) edge (1,0);
\end{tikzpicture}
\end{center}
There is only one nontrivial vertex group, one edge group, and one half-edge group.
We choose a subgroup $n\ZZ$ of the vertex group $\ZZ$, for some $n=0,1,\ldots$, and condition
\eqref{edgecon} forces the edge subgroup to be $n'\ZZ\subset n\ZZ$. Then on the half-edge,
we must choose a subgroup $2n''\ZZ\subset 2n\ZZ$. Integrations are therefore in bijection with the poset 
\[
\{(n, n', n'')\in \NN^3 ~:~ n|n'\text{ and } n|n''\}\cup \{(0,0,0)\}.
\]
The partial order is componentwise divisibility, and $(0,0,0)$ is the least element, corresponding 
to the source-simply-connected integration. 
For Hausdorff integrations, first we have the condition that the edge group coincides with the pullbacks 
from the left and right vertices, which are $n\ZZ$ and $\ZZ$, respectively.  This implies $n=1$ and $n'=1$.  secondly, the half-edge group must coincide with $2n\ZZ$, so we have $n'' = 1$.  Therefore,
we conclude that only one of the integrations is Hausdorff, corresponding to the point $(1,1,1)$ in the above set.  This is, of course, the log pair groupoid constructed in~\S\ref{logintegr}.
\end{example}

\subsubsection{Hausdorff log symplectic integrations}\label{lsclass}

Let $U_j$ be a tubular neighbourhood of one connected component $D_j$ of the degeneracy locus of a log symplectic manifold, and choose basepoints $x_j$, $x_{ji}$ and, if $j\in\sE$, $x_{ji'}\in U_j\cap V_{i'}$, as described in~\S\ref{bspt}.

In Appendix~\ref{Appen: LogSymp}, we construct the source-simply-connected groupoid $\til{\Gg}_{U_j}$ 
integrating $T^*_{\pi} U_j$, compute its poset of 
discrete, totally disconnected normal Lie subgroupoids 
(as well as the subposet of closed subgroupoids), and describe the restriction functor
$\mathbf{P}_{U_j}$.  The results of Propositions~\ref{Prop: LocalSympIntegrationHausdorff} and~\ref{Prop: LocalSympIntegrationNorHausdorff} are summarized as follows.

\begin{theorem}[Local classification]\label{locB}
Let $D_j$, $U_j$, and $V_i$ be as in~\S\ref{graphorbit} and choose basepoints as in~\S\ref{bspt}.
\begin{enumerate}
\item
If $ND_j$ is orientable, then the Hausdorff integrations of $T^*_{\pi} U_j$ are classified by pairs 
\begin{equation}\label{clals}
(K_{ji}, K_{ji'}) 
\end{equation}
of normal subgroups $K_{ji}\subset \pi_1(U_j\cap V_i, x_{ji})$ and $K_{ji'}\subset \pi_1(U_j\cap V_{i'}, x_{ji'})$, which are compatible with the projection $r:U_j\setminus D_j\to D_j$ of the punctured tubular neighbourhood and inclusion map $\iota_j: F_j\to D_j$, in the sense
\begin{equation}\label{condisy}
(\iota_j)_*^{-1}(r_*K_{ji}) = (\iota_j)_*^{-1}(r_*K_{ji'}),
\end{equation}
as subgroups of $\pi_1(F_j,x_j)$.

\item If $ND_j$ is non-orientable, then the Hausdorff integrations of $T^*_{\pi} U_j$ are classified by a normal subgroup \(K_{ji}\) as above, with no additional constraint.

\item Morphisms between integrations correspond to componentwise inclusion for the associated pair 
of normal subgroups.

\item The restriction of an integration of $T^*_{\pi} U_j$ given by~\eqref{clals} to $U_j\cap V_i$ is an integration of $TU_j$, obtained in the same way as in Equation~\ref{quotclas}.

\item The fundamental group of the source fibre over the point $x_j\in D_j$ is isomorphic to $(\iota_j)_*^{-1}(r_*K_{ji})$.
\end{enumerate}
\end{theorem}

\begin{remark}\label{identib}
For $D_j$ orientable, $r_*$ is an isomorphism, so we may view the groups~\eqref{clals} as subgroups of the same group $\pi_1(D_j, x_j)$. Condition~\eqref{condisy} is simply that their preimages in $\pi_1(F_j,x_j)$ agree.
\end{remark}


Theorem~\ref{locB} and Equation \ref{commclass} allow us to give an explicit description of the category of Hausdorff integrations $\Gpd^\Hh(T^*_\pi M)$. We will express the coproduct in terms of the graph introduced in~\S\ref{graphorbit}, using the basepoint choices from~\S\ref{bspt}.

\begin{definition}\label{gragros}
The graph of groups associated to a proper log symplectic manifold $(M,\pi)$ is defined as follows. 
Let $\Gamma$ be the graph associated to $(M,\pi)$ in Definition~\ref{graphmd}.  Let $\delta_{ji}$ be paths joining $y_i$ to $x_{ji}$ for all $i\in\sV, j\in\sD$, and let $\iota_j:F_j\to D_j$ be the inclusion of the symplectic leaf through $x_j$. We label the graph with groups and homomorphisms in the following way, using the identifications in Remark~\ref{identib}.  

\begin{enumerate}
\item[--] To each vertex $i\in\sV$, we associate the group $\pi_1(V_i, y_i)$.
\item[--] To each edge $j\in\sE$ joining $i$ to $i'$, we associate the morphism 
of groups $(\iota_j)_*:\pi_1(F_j,x_j)\to \pi_1(D_j, x_j)$, together with the induced homomorphisms 
$(\delta_{ji})_*, (\delta_{ji'})_*$ from $\pi_1(D_j, x_j)$ to the
vertex groups $\pi_1(V_i, y_i)$ and $\pi_1(V_{i'}, y_{i'})$, as below:  

\begin{center}
\centering
\begin{tikzpicture}
\node (above) at (0,1.8) {$\pi_1(F_j)$};
\node (mid) at (0,.8) {$\pi_1(D_j)$};
\node (left) at (-2, .4) {$\pi_1(V_i)$};
\node (right) at (2, .4) {$\pi_1(V_{i'})$};
\node [circle,fill=black,inner sep=1pt] at (2,0) {};
\node [circle,fill=black,inner sep=1pt] at (-2,0) {};
\path (above) edge [->] (mid);
\path (mid) edge [->] (right);
\path (mid) edge [->] (left);
\path (2,0) edge [in=0,out=180] node [] {} (-2,0);
\end{tikzpicture}
\end{center}
\item[--] To each half-edge $j\in\sH$ attached to $i\in\sV$, we associate the inclusion of groups $\ker w^j_1\hookrightarrow \pi_1(D_j,x_j)$ determined by the Stiefel-Whitney class $w^j_1$ of $ND_j$, as well as the morphism $(\iota_j)_*:\pi_1(F_j,x_j)\to \pi_1(D_j,x_j)$, and finally the induced homomorphism $(\delta_{ji})_*:\ker w^j_1\to \pi_1(V_i,y_i)$.
\begin{center}
\begin{tikzpicture}
\node (up) at (-1,1.3) {$\ker w_1^j$};
\node (mid) at (0,.4) {$\pi_1(D_j)$};
\node (left) at (-2, .4) {$\pi_1(V_i)$};
\node (right) at (1, 1.3) {$\pi_1(F_{j})$};
\node [circle,fill=black,inner sep=1pt] at (-2,0) {};
\path (up) edge [->] (mid);
\path (up) edge [->] (left);
\path (right) edge [->] (mid);
\path (0,0) edge [in=0,out=180] node [] {} (-2,0);
\end{tikzpicture}
\end{center}
\end{enumerate}
\end{definition}

\begin{remark}
For a proper log symplectic manifold $(M, \pi)$, the adjoint integration of $T^*_\pi M$ is a symplectic groupoid by Corollary \ref{SymGpdAdjoint}. It follows from Proposition \ref{prop: multisymp} that all other integrations of $T^*_\pi M$ are also symplectic groupoids.
\end{remark}

\begin{theorem}[Global classification]\label{classifylogsymp}
Given the graph of groups from Definition~\ref{gragros}, the category of Hausdorff symplectic groupoids $\Gpd^{\Hh}(T^*_\pi M)$ is equivalent to the poset whose elements consist of a family of normal subgroups $K_i$ of each vertex group $\pi_1(V_i, y_i), i\in\sV$, such that if $i, i'$ share an edge $j\in\sE$, then $K_i, K_{i'}$ coincide upon restriction to $\pi_1(F_j, x_j)$, that is,
\begin{equation*}
(\iota_j)_*^{-1}(\delta_{ji})_*^{-1} K_i = (\iota_j)_*^{-1}(\delta_{ji})_*^{-1} K_{i'}.
\end{equation*}
The partial order is componentwise inclusion for corresponding normal subgroups, and the 
fundamental group of the source fiber over $x_j, j\in\sD$, is given by the restriction
\begin{equation}\label{eq: s-pi1}
(\iota_j)_*^{-1}(\delta_{ji})_*^{-1} K_i,
\end{equation}
for $i$ attached to $j\in\sD$.  
\end{theorem}




\begin{corollary}
The source-simply-connected integration of a proper log symplectic manifold is Hausdorff if and only if, for each symplectic leaf $F$ contained in the degeneracy hypersurface $D$, and for each class $\gamma\in\pi_1(F)$ on which the first Stiefel-Whitney class of $ND$ vanishes, the push-off of $\gamma$ is nonzero in the fundamental group of the adjacent open symplectic leaf or pair of leaves. 
\end{corollary}

\begin{example} The Hausdorff symplectic groupoids of the Poisson structure described in Example~\ref{Example: EllipticCurve} are classified using the following graph of groups:
\begin{center}
\begin{tikzpicture}
	\node (mid) at (-1,.3) {$\ZZ$};
	\node (right) at (1, .3) {$0$};
	\node (up) at (-1.5,1.1) {$\ZZ$};
	\node (left) at (-2, .3) {$\ZZ$};
	\node (F) at (-2.5, 1.1) {$0$};
	\node (top) at (0,1.5) {$0$};
	\node (edge) at (0,.7) {$\ZZ$};
	\node [circle,fill=black,inner sep=1pt] at (-1,0) {};
	\node [circle,fill=black,inner sep=1pt] at (1,0) {};
	\path (up) edge [->] node[very near start, right] {\scriptsize $\cong$} (mid);
	\path (up) edge [->] node[left,very near start] {\scriptsize $2$} (left);
	\path (top) edge [->] (edge);
	\path (edge) edge [->] node[below, near start] {\scriptsize $\cong$} (mid);
	\path (edge) edge [->] (right);
	\path (F) edge [->] (left);
	\path (-2,0) edge (1,0);
\end{tikzpicture}
\end{center}
For any choice of subgroup $n\ZZ\subset \ZZ$ of the only nontrivial vertex group, the conditions of Theorem~\ref{classifylogsymp} are trivially satisfied.  Hence, the integrations are classified by the poset $\NN\cup \{0\}$, 
where the partial order is divisibility and $0$ is the minimum.  Applying~\eqref{eq: s-pi1} to the diagram above, we see that the source fiber over $x\in D$ has trivial fundamental group for any choice of $n$.  Hence $0$ represents the source-simply-connected integration. 
 
A similar argument may be used to construct the Hausdorff source-simply-connected integration of any log symplectic 2-manifold, whose existence was shown in~\cite{MR2592728}. 
\end{example}

\begin{example}
Let $(M,\pi)$ be the log symplectic 4-manifold constructed as a $\ZZ^2$ quotient 
of $\RR^2\times T^2$, equipped with the Poisson structure 
\[
\left(\sin(2\pi y)\del{x}\wedge\del{y}\right) \oplus \omega^{-1}, 
\] 
where $\omega$ is the standard symplectic form on $T^2$. The 
action is given by 
\[
(n_1, n_2):(x,y,p)\mapsto (x+n_1, y+n_2, \varphi^{n_2}(p)),
\]
for some fixed $\varphi\in \mathrm{SL}(2,\ZZ)$. 
The degeneracy locus of $(M, \pi)$ is the union of two mapping tori $D_1, D_2$, each isomorphic to $S^1 \ltimes_\varphi T^2$. The open symplectic leaves $V_1, V_2$ are each homotopic to $S^1 \ltimes_\varphi T^2$. All of these have fundamental group $\ZZ \ltimes_{\varphi_*} \ZZ^2$. A symplectic leaf $F_j\subset D_j$ is isomorphic to $(T^2, \omega)$ and its fundamental group $\pi_1(F_j, *) = \ZZ^2$ is a normal subgroup of $\ZZ \ltimes_{\varphi_*} \ZZ^2$.
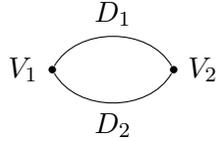
\begin{center}
\begin{tikzpicture}[node distance=1.5cm]
\node [circle,fill=black,inner sep=1pt,label=right:$V_2$] (V2) {};
\node [circle,fill=black,inner sep=1pt,label=left:$V_1$] (V1) [left =of V2] {};

\path (V1) edge [in=120,out=60] node [label=above:$D_1$] {} (V2);
\path (V1) edge [in=240,out=300] node [label=below:$D_2$] {} (V2); 
\end{tikzpicture}
\end{center}
The graph of $(M, \pi)$ is shown above. By Theorem~\ref{classifylogsymp}, each of its Hausdorff symplectic groupoids is given by a pair of normal subgroups $N_1, N_2$ of $\ZZ \ltimes_{\phi_*} \ZZ^2$ such that $N_1 \cap (\ZZ\times\ZZ) = N_2 \cap (\ZZ\times\ZZ)$.  In particular, taking both $N_1, N_2$ to be trivial,~\eqref{eq: s-pi1} yields a trivial fundamental group for the source fibers over $D_1, D_2$, so that we obtain a Hausdorff  source-simply-connected symplectic groupoid.
\end{example}

\appendix

\section{Local normal forms} \label{Appen: Local}

In this appendix, we classify the integrations of the log tangent algebroid $T_DM$ 
and the log symplectic Lie algebroid $T^*_\pi M$, in a tubular neighbourhood of a single
connected component of the hypersurface $D\subset M$ along which the anchor map drops rank.  

This classification is achieved, following Theorem~\ref{Thm: Gpd-NormalSubgpd}, by first constructing the source-simply-connected 
integration, and then classifying its possible discrete, totally disconnected normal Lie subgroupoids. 

We must also take care to describe the restriction of the resulting integrations to the 
punctured tubular neighbourhood, so that the local classification can be used in
Theorem~\ref{Thm: MfldvanKampen} for gluing.

	\subsection{Log tangent case} \label{Appen: LogTang}

	Let $D$ be a connected manifold and $p: N \rightarrow D$ a real line bundle.  We may describe 
the source-simply-connected integration $\Gg$ of $T_D(\tot(N))$ as follows.  The restriction of $T_D(\tot(N))$ to $D\subset \tot(N)$ is the Atiyah algebroid $\At(N)$ of the bundle $N$, which has source-simply-connected integration given by the Holonomy groupoid $\Hol(N)\rra D$, defined by 
\[
\Hol(N) = \{(\gamma, a) ~|~ \gamma\in\Pi_1 D, ~ a: N_{s_0(\gamma)} \stackrel{\cong}{\longrightarrow} N_{t_0(\gamma)}\}.
\]
Moreover, $\Hol(N)$ acts on $N$ and the action groupoid
\begin{equation} \label{eq: ssclogtangpd}
\Gg = \Hol^c(N) \ltimes N\rra \tot(N)
\end{equation}
is the ssc integration of $T_D (\tot N)$.
We divide the subsequent argument into two halves, as $N$ may be orientable or not. We also use $N$ to denote $\tot(N)$, when convenient.
\begin{prop} \label{Prop: MonthubertLocal}
Let $p: N \rightarrow D$ be orientable, and let $x \in D \subset N$ be a basepoint. We denote the two connected components of $N\setminus D$ by $N^+$ and $N^-$, and choose base points $x^\pm \in N^\pm$ such that $p(x^\pm) = x$. We also define $r=p|_{N\setminus D}$.
\begin{enumerate}
\item The integrations of $T_D N$ are classified by triples
\begin{equation} \label{eq: logTang3}
(K^+, K, K^-)
\end{equation}
of normal subgroups $K^+ \subset \pi_1(N^+, x^+)$, $K \subset \pi_1(D, x)$, and $K^- \subset \pi_1(N^-, x^-)$ such that
\[
K \subset r_*K^+ ~\text{ and }~ K \subset r_*K^-.
\]
\item Hausdorff integrations are those triples such that $K = r_*K^+ = r_*K^-$.
\item Morphisms between integrations correspond to componentwise inclusion for the associated triple of normal subgroups.
\item The fundamental group of the source fiber at $x$ is isomorphic to $K$.
\item Restricting the integration of $T_D N$ to $N^\pm$, we obtain the integration of $T(N^\pm)$ given by 
\[
\Pi_1(N^\pm) / \Nn^\pm,
\]
where $\Nn^\pm$ is the unique totally disconnected normal Lie subgroupoid of $\Pi_1(N^\pm)$ with isotropy $K^\pm$ at $x^\pm$. 
\end{enumerate}
\end{prop}

\begin{proof}
Using a trivialization $\tot(N) \cong \RR \times D$ with $N^+ = \RR^+ \times D$, $N^- = \RR^- \times D$, we express the ssc integration $\Gg \rra N$ as in \eqref{eq: ssclogtangpd} explicitly,
\[
\Gg = (\RR^+ \ltimes \RR) \times \Pi_1(D).
\]
By Theorem \ref{Thm: Gpd-NormalSubgpd}, the integrations of $T_D N$ are classified by the discrete, totally disconnected, normal Lie subgroupoids of $\Gg \rra N$, with the Hausdorff integrations requiring the normal subgroupoids be closed.

Let $\Nn$ be a closed discrete, totally disconnected, normal Lie subgroupoid of $\Gg$. Since $\Gg|_{N^+} \cong \Pi_1 (N^+)$, its isotropy $K^+$ at $x^+$ is a normal subgroup of $\pi_1(N^+, x^+)$. Explicitly, we have $\Nn|_{N^+} = \{1\} \times \RR^+ \times \Kk^+$ where $\Kk^+$ is the normal subgroupoid of $\Pi_1 D$ induced by $r_*K^+$. Likewise, we have $\Nn|_{N^-} = \{1\} \times \RR^- \times \Kk^-$ where $\Kk^-$ is the normal subgroupoid of $\Pi_1 D$ induced by $r_*K^-$.

We have $\Nn|_D = \{1\} \times \{0\} \times \Kk$ where $\Kk$ is the normal subgroupoid of $\Pi_1 D$ induced by $K$. Indeed, if 
$\Nn$ contains a point $p = (a, 0, \gamma) \in \Gg$ where $a \neq 1$, then $\Nn$ contains both $\RR^+ \times \{0\} \times \id(D)$ and the identity bisection $\id(N) = \{1\} \times \RR \times \id(D)$, the union of which is not a manifold.

To summarize, we have three normal subgroups $r_*K^-$, $K$ and $r_*K^+$ of $\pi_1(D, x)$ inducing three normal subgroupoids $\Kk^-$, $\Kk$ and $\Kk^+$ of $\Pi_1 D$ respectively; and
\begin{equation}
\Nn = \left(\{1\} \times \RR^- \times \Kk^-\right) \coprod \left(\{1\} \times \{0\} \times \Kk\right) \coprod \left(\{1\} \times \RR^+ \times \Kk^+\right)
\end{equation}
which is a regular submanifold $\Gg$ if and only if $K \subset r_*K^-$ and $K \subset r_*K^+$, obtaining \textit{i)}. Moreover, $\Nn$ is a closed submanifold if and only if $K = r_*K^- = r_*K^+$, obtaining \textit{ii)}. The results \textit{iii)}, \textit{iv)} and \textit{v)} follow from the construction.
\end{proof}

\begin{prop} \label{Prop: MonthubertLocalNOr}

Let $p: N \rightarrow D$ be non-orientable, and choose base points $x \in D$ and $x' \in N\setminus D$ such that $p(x') = x$.  Also, let $r=p|_{N\setminus D}$.
\begin{enumerate}
\item The integrations of $T_D N$ are classified by pairs
\begin{equation} \label{eq: logTang2}
(K', K)
\end{equation}
of normal subgroups $K' \subset \pi_1(N\setminus D, x')$ and $K \subset \pi_1(D, x)$ such that $K \subset r_*K'$.
\item Hausdorff integrations are those pairs such that $K = r_*K'$.
\item Morphisms between integrations correspond to componentwise inclusion for the associated pair of normal subgroups.
\item The fundamental group of the source fiber at $x$ is isomorphic to $K$.
\item Restricting the integration of $T_D N$ to $N\setminus D$, we obtain the integration of $T(N\setminus D)$ given by 
\[
\Pi_1(N\setminus D) / \Nn',
\]
where $\Nn'$ is the unique totally disconnected normal Lie subgroupoid of $\Pi_1(N\setminus D)$ with isotropy $K'$ at $x'$.
\end{enumerate}

\end{prop}

\begin{proof}
The line bundle $N$ determines a double cover $\varsigma: \til{D} \rightarrow D$ and $\til{N} = \varsigma^* N$ is a trivial line bundle, which admits an involution $\tau$ such that $\tot(\til{N}) / \tau = \tot(N)$. Note that $\tau$ induces an involution on the ssc integration $\til{G} \rra \til{N}$ of $T_{\til{D}} \til{N}$. The ssc integration of $T_D N$ is the quotient
\[
(\til{\Gg} \rra \til{N}) / \tau.
\]
We apply Proposition \ref{Prop: MonthubertLocal} to $T_{\til{D}} \til{N}$ and seek the $\tau$-invariant, discrete, totally disconnected, normal subgroupoids. Using the same notation as in Proposition \ref{Prop: MonthubertLocal}, these subgroupoids are classified by triples $(\til{K}^+, \til{K}, \til{K}^-)$ of normal subgroups such that $\til{K} \subset \til{r}_*\til{K}^+$ and $\til{K} \subset \til{r}_*\til{K}^-$, and $\tau$-invariance imposes $\til{K}^+ = \til{K}^-$. If we identify $\pi_1(\til{D}, \til{x})$ with $\pi_1(N\setminus D, x')$, define $K' = \til{r}_*\til{K}^+$ and $K = r_*\til{K}$, then $\til{K} \subset \til{r}_*\til{K}^+$ implies $K \subset r_*K'$ obtaining \textit{i)}. On the other hand, the closed condition $\til{K} = \til{r}_*\til{K}^+$ implies $K = r_*K'$ obtaining \textit{ii)}. The results \textit{iii)}, \textit{iv)} and \textit{v)} follow from the construction.

\end{proof}

	\subsection{Log symplectic case} \label{Appen: LogSymp}

		Following Proposition \ref{Prop: LogSympLineariation}, the local normal form of a proper log symplectic structure near a connected component of its degeneracy locus is built from the following data:
	\begin{itemize}
		\item[--] A compact, connected, symplectic manifold $(F, \omega)$, a symplectomorphism $\varphi: F \rightarrow F$, and a constant $\lambda\in \RR^+$, which determine the Poisson mapping torus 
		\[
		D = S^1_\lambda \ltimes_\varphi F
		\] 
		as defined in~\eqref{maptor}, with projection $f:D\to S^1_\lambda$;
		\item[--] A line bundle $L$ over $D$ induced by an $\ZZ$-equivariant line bundle over $F$ with a metric connection $\nabla$;
		\item[--] An orientable or non-orientable line bundle over $S^1_\lambda$ which we call $Q^+, Q^-$, respectively.
	\end{itemize}
Then the total space of $N = f^*Q^\pm\otimes L$ inherits a log symplectic structure $\pi$, as explained in Proposition~\ref{totpois}.
We construct the ssc symplectic groupoid of $(\tot(N), \pi)$ as an action groupoid of the fiber product of two groupoids.

The first groupoid is the monodromy groupoid, obtained by lifting $\varphi: F \rightarrow F$ to $\varphi: \Pi_1 F \rightarrow \Pi_1 F$:
	\[
		\Mon(D,f) = S^1_\lambda \ltimes_\varphi \Pi_1 F = \frac{\Pi_1 F \times \RR}{(\gamma,t)\sim (\varphi(\gamma), t + \lambda)}.
	\]
This is a Lie groupoid over $D$, and using the metric connection $\nabla$ on $L$, we obtain an action of $\Mon(D,f)$ on $L$.

In the case that the orientable line bundle $Q^+$ is chosen,  the second groupoid, $\Aa^+$, is defined to be the trivial bundle of groups $\Aa^+ = A \times S^1_{\lambda}$, where $A = \RR^+ \ltimes \RR$ is the group of affine transformations of the plane. Using a trivialization $Q^+ = S^1_\lambda \times \RR$ with coordinates $(t, r)$, we obtain an action of $\Aa^+$ on $Q^+$ via
\begin{equation}\label{aact}
\begin{pmatrix}t\\r\end{pmatrix}\mapsto \begin{pmatrix} 1 & b \\ 0 & a \end{pmatrix}\begin{pmatrix}t\\r\end{pmatrix}, \text{ for } \begin{pmatrix} 1 & b \\ 0 & a \end{pmatrix} \in A.
\end{equation}
In the case that $Q^-$ is chosen, we define a groupoid $\Aa^-$ by taking the quotient of $A \times S^1_{2\lambda}$ by the involution $\sigma$ defined by
	\[
	\begin{aligned}
		 A \times S^1_{2\lambda} &\xrightarrow{\sigma} A \times S^1_{2\lambda}, \\
		 \left(\begin{pmatrix} 1 & b \\ 0 & a \end{pmatrix}, t \right) & \mapsto \left(\begin{pmatrix} 1 & -b \\ 0 & a \end{pmatrix}, t +\lambda \right).
	\end{aligned}
	\]
Then $\Aa^- = (A \times S^1_{2\lambda}) / \sigma$ is a nontrivial bundle of groups over $S^1_\lambda$.  By expressing $Q^-$ as the quotient $(t,r)\sim (t+\lambda, -r)$, it 
inherits an $\Aa^-$--action as in~\eqref{aact}. 

Having the groupoid $\Mon(D,f)$ over $D$, and pulling back $\Aa^\pm$ to a groupoid over $D$, we form the fiber product groupoid
\[
\Hh = f^*\Aa^\pm  \times_{D} \Mon(D, f),
\]
obtaining a Lie groupoid over $D$ which acts on $N = f^*Q^\pm \otimes L$ by combining the action of $\Mon(D, f)$ on $L$ and the action of $\Aa^\pm$ on $Q^\pm$. Finally, the action groupoid
	\begin{equation} \label{eq: ssclogsympgpd}
		\Gg = \Hh \ltimes N
	\end{equation}
is the source-simply-connected integration of the Poisson algebroid $T^*_\pi(\tot(N))$. 


\begin{prop} \label{Prop: LocalSympIntegrationHausdorff}
	Let $p: N \rightarrow D$, as defined above, be orientable, let $x\in D$ be a basepoint on the zero section, and let $\iota: F \hookrightarrow D$ be the inclusion of a symplectic leaf through $x$. We denote the two connected components of $N\setminus D$ by $N^+$ and $N^-$,
	and choose base points $x^\pm \in N^\pm$ such that $p(x^\pm) = x$. Also, we let 
	$r = p|_{N\setminus D}$.
	\begin{enumerate}
		\item The Hausdorff integrations of $T^*_\pi N$ are classified by pairs
			\begin{equation} \label{eq: logSymp2}
				(K^+, K^-)
			\end{equation}
		of normal subgroups $K^+ \subset \pi_1(N^+, x^+)$, and $K^- \subset \pi_1(N^-, x^-)$ such that 
			\[
				\iota_*^{-1}(r_*K^+) = \iota_*^{-1}(r_*K^-).
			\]
		\item Morphisms between integrations correspond to componentwise inclusion for the associated pair of normal subgroups.
		\item The fundamental group of the source fiber at $x$ is isomorphic to $\iota_*^{-1}(r_*K^+)$.
		\item 
		Restricting the integration of $T^*_\pi N$ to $N^\pm$ and $N^\pm$, we obtain the integration of $T(N^\pm)$
\[
\Pi_1(N^\pm) / \Nn^\pm
\]
where $\Nn^\pm$ is the unique totally disconnected normal Lie subgroupoid of $\Pi_1(N^\pm)$ with isotropy $K^\pm$ at $x^\pm$.
	\end{enumerate}
\end{prop}

\begin{proof}
Using a trivialization, we may decompose $\tot(N)$ as follows:
\begin{align*}
\tot(N) & = (S^1_\lambda \ltimes_\varphi F) \times \RR \\
& = \{(t, x, r) ~|~ t\in S^1_\lambda, ~x\in F, ~r \in \RR, ~(t, \gamma, r) \sim (t+\lambda, \varphi(\gamma), r)\}.
\end{align*}
Similarly, we write the ssc integration $\Gg \rra N$ defined in~\eqref{eq: ssclogsympgpd} explicitly:
\begin{equation} \label{eq: expl G}
\begin{aligned}
\Gg & = (\Aa^+ \times_D (S^1_\lambda \ltimes_\varphi \Pi_1 F)) \ltimes N \\
& = \{(a, b, t, \gamma, r) ~|~ a\in \RR^+, ~b\in \RR, ~t\in S^1_\lambda, ~\gamma \in \Pi_1 F, ~r \in \RR, \\
& ~~~~~~~~~~~~~~~~~~~~~~~ (a, b, t, \gamma, r) \sim (a, b, t+\lambda, \varphi(\gamma), r)\},
\end{aligned}
\end{equation}
where the source and target maps to $\tot(N)$ are given by
\[
s: (a, b, t, \gamma, r) \mapsto (t, s_0(\gamma), r), ~~ t: (a, b, t, \gamma, r) \mapsto (t+br, t_0(\gamma), ar).
\]
By Theorem \ref{Thm: Gpd-NormalSubgpd}, it suffices to classify closed, discrete, totally disconnected, normal Lie subgroupoids of $\Gg \rra \tot(N)$.  So, let $\Nn$ be such a subgroupoid.

First note that the isotropy groups of the groupoid $\Gg$ are given as follows: at $x^+ = (t, x, r) \in N \setminus D$, we have $r > 0$ and the isotropy group of $\Gg$ is given by
\begin{equation} \label{eq: normal-iso'}
\Gg(x', x') = \left(\left(\{1\} \times \tfrac{\lambda}{r}\ZZ \times \{t\}\right) \ltimes \pi_1(F, x)\right) \times \{r\},
\end{equation}
while at $x = (t, x, 0) \in D$, the isotropy group of $\Gg$ is given by
\begin{equation} \label{eq: normal-iso}
\Gg(x, x) = \left(\left(\RR^+ \ltimes \RR) \times \{t\}\right) \times \pi_1(F, x)\right) \times \{0\}.
\end{equation}

For a point $p = (a, b, t, \gamma, 0) \in \Gg|_D$ such that $a \neq 1$ or $b \neq 0$, if we take a small neighbourhood $U_p$ around $p$, then $U_p \cap \Nn|_{N\setminus D} = \varnothing$. Since $\dim \Nn = \dim M$, it follows that $\dim \left(U_p \cap \Nn|_D\right) = \dim M$. However since $s: \Nn|_D \rightarrow D$ is a submersion, we must have that $\dim \left(\Nn \cap s^{-1}(s(p))\right) = 1$, which is an contradiction.

Therefore the subgroupoid $\Nn|_D$ must take the form 
\[
\Nn|_D = \left((\{1\} \times \{0\} \times S^1_\lambda) \ltimes \Hh\right) \times \{0\}.
\]
where $\Hh$ is the normal subgroupoid $\Pi_1 F$ induced by a normal subgroup $H \subset \pi_1(F, x)$. We denote the isotropy group of $\Nn$ at $x^+$ by $K^+$. By \eqref{eq: normal-iso'}, $K^+$ is a normal subgroup of $\tfrac{\lambda}{r}\ZZ \ltimes \pi_1(F, x)$. If we define $H^+ = K^+ \cap \pi_1(F, x) = \iota_*^{-1}(r_*K^+)$, which is a $\varphi$-invariant normal subgroup of $\pi_1(F, x)$, then as we take the limit $r \rightarrow 0$, the condition that $\Nn$ is closed implies $H^+ = H$. Similarly, we have $H^- = H$.

Since $H^+ = H = H^-$, we conclude that
\[
\iota_*^{-1}(r_*K^+) = \iota_*^{-1}(r_*K^-),
\]
obtaining \textit{i)}. The results \textit{ii)}, \textit{iii)} and \textit{iv)} follow from the construction.
\end{proof}

\begin{prop} \label{Prop: LocalSympIntegrationNorHausdorff}
Let $p: N \rightarrow D$, as defined above, be non-orientable, let $x\in D$ be a basepoint on the zero section, and let $\iota:F\hookrightarrow D$ be the inclusion of a symplectic leaf through $x$.  Choose a base point $x' \in N\setminus D$, and let $r=p|_{N\setminus D}$.
\begin{enumerate}
\item The Hausdorff integrations of $T^*_\pi N$ are classified by a normal subgroup $K' \subset \pi_1(N\setminus D, x')$.
\item Morphisms between integrations correspond to inclusions of associated normal subgroups.
\item The fundamental group of the source fiber at $x$ is isomorphic to $\iota_*^{-1}(r_*K')$.
\item Restricting the integration of $T^*_\pi N$ to $N\setminus D$, we obtain the integration of $T(N\setminus D)$ given by
\[
\Pi_1(N\setminus D) / \Nn'
\]
where $\Nn'$ is the unique closed, totally disconnected, normal Lie subgroupoid of $\Pi_1(N\setminus D)$ with isotropy $K'$ at $x'$.
\end{enumerate}
\end{prop}

\begin{proof}
Using the same strategy as in the proof of Proposition \ref{Prop: MonthubertLocalNOr}, $N$ determines a double cover $\varsigma: \til{D} \rightarrow D$ and an involution $\tau$ on the trivial pullback $\til{N} = \varsigma^* N$, whose total space carries a $\tau$--invariant Poisson structure $\til{\pi}$. This induces an involution on the ssc integration $\til{\Gg} \rra \til{N}$ of $T^*_{\til{\pi}} \til{N}$.

We apply Proposition \ref{Prop: LocalSympIntegrationHausdorff} to $T^*_{\til{\pi}} \til{N}$ and seek the $\tau$-invariant, closed, discrete, totally disconnected, normal subgroupoids. Using the same notation as in Proposition \ref{Prop: LocalSympIntegrationHausdorff}, these subgroupoids are classified by pairs $(\til{K}^+, \til{K}^-)$ such that $\til{\iota}_*^{-1}(\til{r}_*\til{K}^+) = \til{\iota}_*^{-1}(\til{r}_*\til{K}^-)$, and the $\tau$-invariance implies $\til{K}^+ = \til{K}^-$. Identifying $\pi_1(\til{D}, \til{x})$ with $\pi_1(N\setminus D, x')$, and taking $K' = \til{K}^+$, we obtain \textit{i)}. The results \textit{ii)}, \textit{iii)} and \textit{iv)} follow from the construction.
\end{proof}

\bibliographystyle{hyperamsplain} 
\bibliography{SympGpd} 

\end{document}